\RequirePackage{fix-cm}
\documentclass[smallextended,numbook,envcountsame,draft]{svjour3}       
\smartqed  
\usepackage{latexsym}
\usepackage{bbm}
\usepackage{amssymb}
\usepackage{amsmath}
\usepackage{mathtools}
\usepackage{xfrac}
\usepackage{comma}
\usepackage{graphicx}
\usepackage{cutwin}
\usepackage{pstricks}
\usepackage{pst-plot}
\usepackage{pgfplots}

\usepackage{enumerate}

\numberwithin{equation}{section}

\newcommand{\defeq}{\vcentcolon=}
\newcommand{\eqdef}{=\vcentcolon}

\newcommand{\N}{\mathbb{N}}
\newcommand{\Z}{\mathbb{Z}}

\newcommand{\R}{\mathbb{R}}

\newcommand{\V}{\mathbb{V}}

\newcommand{\1}{\mathbbm{1}}
\newcommand{\comp}{\mathsf{c}}

\newcommand{\x}{\mathsf{x}}
\newcommand{\y}{\mathsf{y}}

\newcommand{\Prm}{\mathrm{P}}
\newcommand{\Erm}{\mathrm{E}}
\newcommand{\Varm}{\mathrm{Var}}
\newcommand{\Prmp}{\mathrm{P}_{\!p}}
\newcommand{\Ermp}{\mathrm{E}_{p}}

\newcommand{\Prob}{\mathbb{P}}
\newcommand{\Probi}{\mathrm{P}_{\!\mathit{i}}}
\newcommand{\Probnull}{\mathrm{P}_{\!0}}

\newcommand{\E}{\mathbb{E}}

\newcommand{\Cluster}{\mathcal{C}}
\newcommand{\back}{\mathcal{B}}
\newcommand{\bfnull}{\mathbf{0}}
\newcommand{\Rpre}{\mathcal{R}^{\mathrm{pre}}}
\newcommand{\Rprelambda}{\mathcal{R}^{\mathrm{pre},\lambda}}
\newcommand{\Res}{\mathcal{R}_{\mathrm{eff}}}
\newcommand{\Con}{\mathcal{C}_{\mathrm{eff}}}

\newcommand{\distto}{\stackrel{\mathrm{d}}{\to}}

\newcommand{\F}{\mathcal{F}}
\newcommand{\G}{\mathcal{G}}
\renewcommand{\H}{\mathcal{H}}

\newcommand{\vel}{\overline{\mathrm{v}}}

\newcommand{\lambdacrit}{\lambda_{\mathrm{c}}}
\newcommand{\pesc}{\mathit{p}_{\mathrm{esc}}}

\newcommand{\domega}{\mathrm{d}\omega}

\usepackage{changes}

\newcommand{\auskommentiert}[1]{\textcolor{cyan}{}}

\journalname{Arxiv}
\begin{document}

\title{Einstein relation for random walk in a one-dimensional percolation model
}

\titlerunning{Einstein relation for random walk in a one-dimensional percolation model}        

\author{
Nina Gantert		\and
Matthias Meiners	\and
Sebastian  M\"uller 
}


\institute{Nina Gantert \at
		Fakult\"at f\"ur Mathematik,
		Technische Universit\"at M\"unchen,
		85748 Garching bei M\"unchen, Germany.
		\email{gantert@ma.tum.de}		
		\and
		Matthias Meiners \at
		Institut f\"ur Mathematik, Universit\"at Innsbruck, 6020 Innsbruck, Austria.
		\email{matthias.meiners@uibk.ac.at}
		\and
		Sebastian M\"uller \at
		Aix Marseille Universit\'e, CNRS, Centrale Marseille, I2M UMR 7373, 13453, Marseille, France.
		\email{sebastian.muller@univ-amu.fr}
}

\date{Received: \today / Accepted: date}

\maketitle

\begin{abstract}
We consider random walks on the infinite cluster of a conditional bond percolation model on the infinite ladder graph.
In a companion paper, we have shown that if the random walk is pulled to the right by a positive bias $\lambda > 0$,
then its asymptotic linear speed $\vel$ is continuous in the variable $\lambda > 0$ and differentiable for all sufficiently small $\lambda > 0$.
In the paper at hand, we complement this result by proving that $\vel$ is differentiable at $\lambda = 0$.
Further, we show the Einstein relation for the model, i.e., that the derivative of the speed at $\lambda = 0$
equals the diffusivity of the unbiased walk.

\end{abstract}

\section{Introduction}
We continue our study of regularity properties in \cite{Gantert+al:2018} of a biased random walk on an infinite one-dimensional percolation cluster
introduced by Axelson-Fisk and H\"aggstr\"om \cite{Axelson-Fisk+H"aggstr"om:2009}.
The model was introduced as a tractable model that exhibits similar phenomena as biased random walk on the supercritical percolation model in $\Z^{d}$.
The bias, whose strength is given by some parameter $\lambda>0$, favors the walk to move in a pre-specified direction.  

There exists a critical bias $\lambdacrit \in (0,\infty)$ such that for $\lambda\in (0, \lambdacrit)$ the walk has positive speed while for $\lambda \geq \lambdacrit$ the speed is zero,
see  Axelson-Fisk and H\"aggstr\"om \cite{Axelson-Fisk+H"aggstr"om:2009b}.
The reason for the existence of these two different regimes is that the percolation cluster contains traps (or dead ends) and the walk faces two competing effects.
When the bias becomes larger the time spent in such traps (peninsulas stretching out in the direction of the bias) increases
while the time spent on the backbone (consisting of infinite paths in the direction of the bias) decreases.
Once the bias is sufficiently large the expected time the walk stays in a typical trap is infinite and the speed of the walk becomes zero.

Even though the model may be considered as one of the easiest
non-trivial models for a random walk on a percolation cluster, explicit calculation for the speed $\vel=\vel(\lambda)$ could not be
carried out. The main result of our previous work \cite{Gantert+al:2018}  is that the speed (for fixed
percolation parameter $p$) is continuous in $\lambda$ on $(0,\infty)$.
The continuity  of the speed may seem obvious, but to our best knowledge,
it has not been proved for a biased random walk on a percolation cluster, and not even for biased random walk on Galton-Watson trees.
Moreover, we proved in \cite{Gantert+al:2018} that the speed is differentiable in
$\lambda$ on $(0,\lambdacrit/2)$ and we characterized the derivative as the
covariance of a suitable two-dimensional Brownian motion.

This paper studies the regularity of the speed in $\lambda=0$.
In particular, we establish the Einstein relation for the model: we prove that $\vel$ is differentiable at $\lambda = 0$
and that the derivative at $\lambda = 0$ equals the variance of the scaling limit of the unbiased walk.

The Einstein relation is conjectured to be true in general for reversible motions which behave diffusively.
We refer to Einstein \cite{Einstein} for a historical reference and to Spohn \cite{Spohn} for further explanations.
A weaker form of the Einstein relation holds indeed true under such general assumptions and goes back to Lebowitz and Rost \cite{Lebowitz+Rost:94}.
However, the Einstein relation in the stronger form as described above was only established (or disproved) in examples.
For instance, Loulakis \cite{Loulakis:2002,Loulakis:2005} considers a tagged particle in an exclusion process,
Komorowski and Olla \cite{Komorowski+Olla:2005} and Avena, dos Santos and V\"ollering \cite{Avena+2013} investigate other examples of space-time environments. 

Komorowski and Olla  \cite{Komorowski+Olla:2005a} treat a first example of random walks with random conductances on $\Z^{d}, d\geq 3$,
and Gantert, Guo and Nagel \cite{Gantert+Guo+Nagel:17}
establish the Einstein relation for random walks among i.i.d.~uniform elliptic random conductances on $\Z^{d}, d\geq 1$.
In  dimension one the Einstein relation can be proved via explicit calculations, see Ferrari {\it et al.}~\cite{Ferrari++:1985}.
There are only few results for non-reversible situations, see Guo \cite{Guo:16} and Komorowski and Olla \cite{Komorowski+Olla:2005}.
We want to stress that while the differentiability of the speed might appear as natural  or obvious,
there are examples where the speed is not differentiable, see Faggionato and Salvi \cite{Faggionato+Salvi:18}.

Despite this recent progress not much is known in models with \emph{hard traps}, e.g.\ random conductances without uniform ellipticity condition or percolation clusters.
The first result in this direction is  Ben Arous {\it et al.}~\cite{BenArous+Hu+Olla+Zeitouni:2013} that proves the Einstein relation for  certain biased random walks on Galton-Watson trees. 
An additional difficulty in our model is that the traps are not only \emph{hard} but do also have an influence on the structure of the \emph{backbone}.
Our paper is the first, to our knowledge, to prove the  Einstein relation for a model with hard traps and dependence of traps and backbone.
Although the structure of the traps is elementary the decoupling of traps and backbone is one of the major difficulties we encountered.

We prove a quenched (joint) functional limit theorem via the corrector method, see Section \ref{sec:corrector},
with additional moment bounds for the distance of the walk from the origin.
The law of the unbiased walk is compared with the law of the biased one using a Girsanov transform.
The difference between these measures is quantified using the above joint limit theorem.
Finally, we use regeneration times that depend on the bias and appropriate double limits to conclude
that the derivative of the speed equals a covariance, see Section \ref{sec:proof einstein}.
It remains then to identify the covariance as the variance of the unbiased walk, see Equation \eqref{eq:E[BM]=E[B^2]}.

\section{Preliminaries and main results}	\label{sec:main results}

In this section we introduce the percolation model.
The reader is referred to Figure \ref{fig:conditional percolation} for an illustration.

\paragraph{Percolation on the ladder graph.}
Let $\mathcal{L} = (V,E)$ be the infinite ladder graph
with vertex set $V = \Z \times \{0,1\}$ and edge set $E = \{\langle v,w\rangle: v,w \in V, |v-w|=1\}$
where $\langle v,w\rangle$ is an unordered pair
and $|\cdot|$ the standard Euclidean norm in $\R^2$.
We also write $v \sim w$ for ${\langle v,w\rangle \in E}$, and say that $v$ and $w$ are neighbors.

Axelson-Fisk and H\"aggstr\"om \cite{Axelson-Fisk+H"aggstr"om:2009} introduced
a percolation model on $\mathcal{L}$
that may be called `i.\,i.\,d.~bond percolation on the ladder graph conditioned on the existence of a bi-infinite path'.
We give a short review of this model.

Let $\Omega \defeq \{0,1\}^E$. We call $\Omega$ the \emph{configuration space},
its elements $\omega \in \Omega$ are called \emph{configurations}.
A path in $\mathcal{L}$ is a finite or infinite sequence of distinct edges connecting a finite or infinite sequence of neighboring vertices.
For a given $\omega \in \Omega$,
we call a path $\pi$ in $\mathcal{L}$ \emph{open} if $\omega(e)=1$ for each edge $e$ from $\pi$.
If $\omega \in \Omega$ and $v \in V$,
we denote by $\Cluster_{\omega}(v)$ the connected component in $\omega$ that contains $v$,
i.\,e.,
\begin{equation*}
\Cluster_{\omega}(v)	= \{w \in V: \text{there is an open path in } \omega \text{ connecting } v \text{ and } w\}.
\end{equation*}
We write $\x: V \to \Z$ and $\y:V \to \{0,1\}$ for the projections from $V$ to $\Z$ and $\{0,1\}$, respectively.
Then $v = (\x(v),\y(v))$ for every $v \in V$.
We call $\x(v)$ and $\y(v)$ the $\x$- and $\y$-coordinate of $v$, respectively.
For $N_1, N_2 \in \N$, let $\Omega_{N_1,N_2}$ be the set of configurations
in which there exists an open path from some $v_1 \in V$ with $\x(v_1)=-N_1$ to some $v_2 \in V$ with $\x(v_2) = N_2$.
Further, let $\Omega^* \defeq \bigcap_{N_1, N_2 \geq 0} \Omega_{N_1,N_2}$
be the set of configurations which have an infinite path connecting $-\infty$ and $+\infty$.

Denote by $\F$ the $\sigma$-field on $\Omega$ generated by the projections $p_e: \Omega \to \{0,1\}$,
$\omega \mapsto \omega(e)$, $e \in E$.
For $p \in (0,1)$, let $\mu_p$ be the probability distribution on $(\Omega,\F)$
which makes $(\omega(e))_{e \in E}$ an independent family of Bernoulli variables
with $\mu_p(\omega(e)=1)=p$ for all $e \in E$.
Then $\mu_p(\Omega^*)=0$ by the Borel-Cantelli lemma.
Write $\Prm_{p,N_1,N_2}(\cdot) \defeq \mu_p(\cdot \cap \Omega_{N_1,N_2})/\mu_p(\Omega_{N_1,N_2})$
for the probability distribution on $\Omega$
that arises from conditioning on the existence of an open path from $\x$-coordinate $-N_1$ to $\x$-coordinate $N_2$.
The measures $\Prm_{p,N_1,N_2}(\cdot)$ converge weakly as $N_1,N_2 \to \infty$
as was shown in \cite[Theorem 2.1]{Axelson-Fisk+H"aggstr"om:2009}:

\begin{theorem}
The distributions $\Prm_{p,N_1,N_2}$ converge weakly as $N_1,N_2 \!\to\! \infty$
to a probability measure $\Prmp^*$ on $(\Omega,\F)$ with $\Prmp^*(\Omega^*)=1$.
\end{theorem}
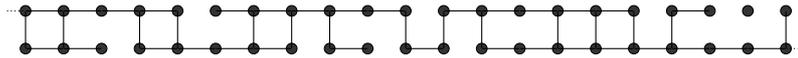
\begin{figure}[htp]
\begin{center}
\pgfmathsetseed{845133}
\begin{tikzpicture}[thin, scale=0.5,-,
                   shorten >=0pt+0.5*\pgflinewidth,
                   shorten <=0pt+0.5*\pgflinewidth,
                   every node/.style={circle,
                                      draw,
                                      fill          = black!80,
                                      inner sep     = 0pt,
                                      minimum width =4 pt}]

\def \p {0.75}

\foreach \x in {-10,-9,-8,-7,-6,-5,-4,-3,-2,-1,0,1,2,3,4,5,6,7,8,9,10}
\foreach \y in {0,1}
    \node at (\x,\y) {};

\foreach \x in {-10,-9,-8,-7,-6,-5,-4,-3,-2,-1,0,1,2,3,4,5,6,7,8,9}{
\foreach \y in {0,1}{
    \pgfmathparse{rnd}
    \let\dummynum=\pgfmathresult
    \ifdim\pgfmathresult pt < \p pt\relax \draw (\x,\y) -- (\x+1,\y);\fi
  }}
  
\foreach \x in {-10,-9,-8,-7,-6,-5,-4,-3,-2,-1,0,1,2,3,4,5,6,7,8,9,10}{
\foreach \y in {0}{
    \pgfmathparse{rnd}
    \let\dummynum=\pgfmathresult
    \ifdim\pgfmathresult pt < \p pt\relax \draw (\x,\y) -- (\x,\y+1);\fi
  }}

	\draw[densely dotted] (-10.5,1) -- (-10,1);
	\draw (-3,1) -- (-2,1);
	\draw (9,0) -- (10,0);
	\draw[densely dotted] (10,0) -- (10.5,0);
\end{tikzpicture}
\end{center}
\caption{A piece of the cluster sampled according to $\Prmp^*$.}
\label{fig:conditional percolation}
\end{figure}
For any $\omega \in \Omega^*$, we write $\Cluster = \Cluster_\omega$ for the $\Prmp^*$-a.\,s.\ unique infinite open cluster.
We write $\mathbf{0} \defeq (0,0)$ und define $\Omega_{\mathbf{0}} \defeq \{\omega \in \Omega^*: \mathbf{0} \in \Cluster\}$ and
\begin{equation*}
\Prmp(\cdot) \defeq \Prmp^*(\cdot | \Omega_{\mathbf{0}}).
\end{equation*}

\paragraph{Random walk on the infinite percolation cluster.}

Throughout the paper,
we keep $p \in (0,1)$ fixed and consider random walks in a percolation environment sampled according to $\Prmp$.
The model to be introduced next goes back to Axelson-Fisk and H\"agg\-str\"om \cite{Axelson-Fisk+H"aggstr"om:2009b},
who used a different parametrization.

We work on the space $V^{\N_0}$
equipped with the $\sigma$-algebra $\G = \sigma(Y_n: n \in \N_0)$
where $Y_n: V^{\N_0} \to V$ denotes the projection from $V^{\N_0}$ onto the $n$th coordinate.
Let $P_{\omega,\lambda}$ be the distribution on $V^{\N_0}$ that makes $Y \defeq (Y_n)_{n \in \N_0}$
a Markov chain on $V$ with initial position $\mathbf{0}$
and transition probabilities $p_{\omega,\lambda}(v,w)=P_{\omega,\lambda}(Y_{n+1} = w \mid Y_n = v)$
defined via
\begin{equation}	\label{eq:P_xi lazy}
p_{\omega,\lambda}(v,w)
=
\begin{cases}
\frac{e^{\lambda (\x(w)-\x(v))}}{e^{\lambda}+1+e^{-\lambda}} \1_{\{\omega(\langle v,w \rangle)=1\}}				&	\text{if } v \sim w,	\\
\sum_{u \sim v} \frac{e^{\lambda (\x(u)-\x(v))}}{e^{\lambda}+1+e^{-\lambda}} \1_{\{\omega(\langle u,v \rangle)=0\}}	&	\text{if } v=w,	\\
0																						&	\text{otherwise}.
\end{cases}
\end{equation}
We write $P^{\mathbf{0}}_{\omega,\lambda}$ to emphasize the initial position $\mathbf{0}$,
and $P^v_{\omega,\lambda}$ for the distribution of the Markov chain with the same transition probabilities but initial position $v \in V$.
The joint distribution of $\omega$ and $(Y_n)_{n \in \N_0}$ on $(\Omega \times \V^{\N_0},\F \otimes \G)$
when $\omega$ is drawn at random according to a probability measure $Q$ on $(\Omega,\F)$
is denoted by $Q \times P^v_{\omega,\lambda} \eqdef \Prob_{Q,\lambda}^v$ where $v$ is the initial position of the walk.
(Notice that, in slight abuse of notation, we consider $Y_n$ also as a mapping from $\Omega \times V^{\N_0}$
to $V$.)
We refer to \cite{Gantert+al:2018} for a formal definition.
We write $\Prob_{\lambda}^v$ for $\Prob_{\Prmp,\lambda}^v$,
$\Prob_{\lambda}$ for $\Prob_{\lambda}^{\mathbf{0}}$
and $\Prob^*_{\lambda}$ for $\Prob_{\Prmp^*,\lambda}^{\mathbf{0}}$.
If the walk starts at $v=\bfnull$, we sometimes omit the superscript $\bfnull$.
Further, if $\lambda=0$, we sometimes omit $\lambda$ as a subscript,
and write $p_{\omega}$ for $p_{\omega,0}$, and $\Prob$ for $\Prob_{0}$.

\paragraph{The speed of the random walk.}

Axelson-Fisk and H\"aggstr\"om \cite[Proposition 3.1]{Axelson-Fisk+H"aggstr"om:2009b}
showed that $(Y_n)_{n \in \N_0}$ is recurrent under $\Prob_0$ and transient under $\Prob_{\lambda}$ for $\lambda \not = 0$.
Moreover, there is a critical bias $\lambdacrit \in (0,\infty)$ separating the ballistic from the sub-ballistic regime.
More precisely, if one denotes by $X_n \defeq \x(Y_n)$ the projection of $Y_n$ on the $\x$-coordinate, the following result holds.

\begin{proposition}	\label{Prop:SLLN}
For any $\lambda \geq 0$, there exists a deterministic constant $\vel(\lambda) = \vel(p,\lambda) \in [0,1]$ such that
\begin{equation*}	\textstyle
\frac{X_n}{n}	~\to~	 \vel(\lambda)	\quad	\Prob_{\lambda} \text{-a.\,s.\ as } n \to \infty.
\end{equation*}
Further, there is a critical bias $\lambdacrit = \lambdacrit(p) > 0$
(for which an explicit expression is available) such that
\begin{equation*}
\vel(\lambda) > 0	\text{ for } 0 < \lambda < \lambdacrit
\quad	\text{ and }	\quad
\vel(\lambda) = 0	\text{ for } \lambda = 0 \text{ and } \lambda \geq \lambdacrit.
\end{equation*}
\end{proposition}
\begin{proof}
For $\lambda > 0$ this is Theorem 3.2 in \cite{Axelson-Fisk+H"aggstr"om:2009b}.
For $\lambda = 0$, the sequence of increments $(X_n-X_{n-1})_{n \in \N}$ is ergodic
by Lemma \ref{Lem:ergodic theory input} below.
Birkhoff's ergodic theorem implies
$\vel(0) = \lim_{n \to \infty} \frac{X_n}{n} = \E[X_1] = 0$ $\Prob$-a.\,s.
\end{proof}

\paragraph{Functional central limit theorem for the unbiased walk.}

In a preceding paper \cite{Gantert+al:2018},
we have shown that $\vel$ is differentiable as a function of $\lambda$ on the interval $(0,\lambdacrit/2)$,
and continuous on $(0,\infty)$.
In this paper, we show that $\vel$ is also differentiable at $0$ with $\vel'(0) = \sigma^2$
where $\sigma^2$ is the limiting variance of $n^{-1/2} X_n$ under the distribution $\Prob_0$.
This is the Einstein relation for the model.
Clearly, a necessary prerequisite for the Einstein relation
is the central limit theorem for the unbiased walk.

Before we provide the central limit theorem for the unbiased walk,
we introduce some notation.
As usual, for $t \geq 0$, we write $\lfloor t \rfloor$ for the largest integer $\leq t$.
Then, we define
\begin{equation*}	\textstyle
B_n \defeq \big(\frac{X_{\lfloor nt\rfloor}}{\sqrt{n}}\big)_{0 \leq t \leq 1}
\end{equation*}
for each $n \in \N$.
The random function $B_n$ takes values in the Skorokhod space $D([0,1])$
of right-continuous real-valued functions with existing left limits.
We denote by ``$\Rightarrow$'' convergence in distribution of random variables in the Skorokhod space $D[0,1]$,
see \cite[Chapter 3]{Billingsley:1968} for details.

\begin{theorem}	\label{Thm:FCLT}
There exists a constant $\sigma = \sigma(p) \in (0,\infty)$ such that
\begin{equation}	\label{eq:invariance principle}	\textstyle
B_n	\Rightarrow	\sigma B
\quad	\text{under } P_{\omega}
\end{equation}
for $\Prmp$-almost all $\omega \in \Omega$ where $B$ is a standard Brownian motion.
\end{theorem}

It is worth mentioning that an annealed invariance principle for $(B_n)_{n \in \N}$ follows
without much effort from \cite{DeMasi+al:1989}.
In principle, we do not require a quenched central limit theorem for the proof of the Einstein relation.
However, we do require a joint central limit theorem for $B_n$ together with a certain martingale $M_n$,
see Theorem \ref{Thm:joint CLT} below.
Therefore, we cannot directly apply the results from \cite{DeMasi+al:1989}.
On the other hand, in the proof of the Einstein relation we use precise bounds on the corrector.
Using similar arguments as Berger and Biskup \cite{Berger+Biskup:2007},
these bounds almost immediately give the quenched invariance principle.

\paragraph{Einstein relation.}

We are now ready to formulate the Einstein relation:
\begin{theorem}	\label{Thm:Einstein relation}
The speed $\vel$ is differentiable at $\lambda=0$ with derivative
\begin{equation}	\label{eq:Einstein relation}	\textstyle
\vel'(0) = \lim_{\lambda \downarrow 0} \frac{\vel(\lambda)}{\lambda} = \sigma^2
\end{equation}
where $\sigma^2$ is given by Theorem \ref{Thm:FCLT}.
\end{theorem}

\paragraph{The joint functional central limit theorem.}

As in \cite{Gantert+al:2018}, the proof of the differentiability of the speed
is based on a joint central limit theorem for $X_n$
and the leading term of a suitable density.

To make this precise, we first introduce some notation.
For $v \in V$, let $N_{\omega}(v) \defeq \{w \in V: p_{\omega,0}(v,w) > 0\}$.
Thus, $N_{\omega}(v) \not = \varnothing$, even for isolated vertices.
For $w \in N_{\omega}(v)$, the function $\log p_{\omega,\lambda}(v,w)$ is differentiable at $\lambda=0$.
Hence, we can write a first-order Taylor expansion of $\log p_{\omega,\lambda}(v,w)$ around $\lambda = 0$ in the form
\begin{equation}	\label{eq:Taylor expansion}
\log p_{\omega,\lambda}(v,w)
= \log p_{\omega}(v,w)
+ \lambda \nu_{\omega}(v,w)+ \lambda o(\lambda)
\end{equation}
where $\nu_{\omega}(v,w)$ is the derivative of  $\log p_{\omega,\lambda}(v,w)$ at $0$
and $o(\lambda)$ converges to $0$ as $\lambda\to0$.
Since there is only a finite number of $1$-step transition probabilities,
$o(\lambda) \to 0$ as $\lambda \to 0$ uniformly (in $v$, $w$ and $\omega$).

For all $v$ and all $\omega$, $p_{\omega,\lambda}(v,\cdot)$ is a probability measure
on $N_{\omega}(v)$ and hence
\begin{equation*}	\textstyle
\sum_{w\in N_{\omega}(v)} \nu_{\omega}(v,w) p_{\omega}(v,w) = 0.
\end{equation*}
Therefore, for fixed $\omega$,
the sequence $(M_{n})_{n \geq 0}$ where $M_{0}=0$ and,
for $n \in \N$,
\begin{equation*}	\textstyle
M_{n}=\sum_{k=1}^{n} \nu_{\omega}(Y_{k-1},Y_{k})
\end{equation*}
is a $P_{\omega}$-martingale with respect to the canonical filtration of the walk $(Y_k)_{k \in \N_0}$.
Clearly, $M_n$ is a (measurable) function of $\omega$ and $(Y_k)_{k \in \N_0}$
and thus a random variable on $\Omega \times V^{\N_0}$.
The sequence $(M_n)_{n \geq 0}$ is also a martingale under the annealed measure $\Prob$.

\begin{theorem}	\label{Thm:joint CLT}
Let $p \in (0,1)$. Then, for $\Prmp$-almost all $\omega \in \Omega$,
\begin{equation}	\label{eq:joint invariance principle}
(B_n(t),n^{-1/2} M_{\lfloor nt \rfloor})		\Rightarrow	(B,M)
\quad	\text{under } P_{\omega}
\end{equation}
where $(B,M)$ is a two-dimensional centered Brownian motion with deterministic covariance matrix
$\Sigma = (\sigma_{ij})_{i,j=1,2}$.
Further, it holds that
\begin{equation}	\label{eq:E[BM]=E[B^2]}
\sigma_{12} = \sigma_{21} = \E[B(1)M(1)] = \E[B(1)^2] = \sigma^2.
\end{equation}
\end{theorem}

As the martingale $M_n$ has bounded increments,
the Azuma-Hoeffding inequality \cite[E14.2]{Williams:1991} applies and gives the following
exponential integrability result, see the proof of Proposition 2.7 in \cite{Gantert+al:2018} for details.

\begin{proposition}	\label{Prop:sup exp M}
For every $t > 0$,
\begin{equation}	\label{eq:sup exp M}	\textstyle
\sup_{n \geq 1} \E_{\lambda}[e^{t n^{-1/2} M_n}] < \infty.
\end{equation}
\end{proposition}

We finish this section with an overview of the steps that lead to the proof of Theorem \ref{Thm:Einstein relation}.
\begin{enumerate}
	\item
		In Section \ref{sec:corrector}, we prove the joint central limit theorem, Theorem \ref{Thm:joint CLT}.
		The proof is based on the corrector method, which is a decomposition technique
		in which $Y_n$ is written as a martingale plus a corrector of smaller order.
		The martingale is constructed in Section \ref{sec:harmonic functions}.
		Many arguments are based on the method of the environment seen from the point of view of the walker.
	\item
		In Lemma \ref{Lem:max X_k^2 bound}, we prove that
		\begin{equation}	\label{eq:supsecondmoment}
		\sup_{n \in \N} \frac1n \E\Big[\max_{k=1,\ldots,n} X_k^2\Big] < \infty.
		\end{equation}
		The proof is based on estimates for the almost sure fluctuations of the corrector
		derived in Section \ref{sec:corrector}.
	\item
		Using the joint central limit theorem and \eqref{eq:supsecondmoment}, we show
		in Proposition \ref{Prop:3rd step} that, for any $\alpha > 0$,
		\begin{equation}	\label{eq:speed approx by covariance}
		\lim_{\substack{\lambda \to 0,\\ \lambda^2 n \to \alpha}} \frac{\E_{\lambda}[X_n]}{\lambda n}
		~=~ \E[B(1) M(1)].
		\end{equation}
		Equation \eqref{eq:speed approx by covariance} is a weak form of the Einstein relation going back to Lebo\-witz and Rost \cite{Lebowitz+Rost:94}.
	\item
		Finally, we show in Section \ref{subsec:final step} that
		\begin{equation}	\label{eq:2nd step}
		\lim_{\alpha\to\infty} \lim_{\substack{\lambda \to 0,\\ \lambda^{2} n \to \alpha}}
		\bigg[\frac{\vel(\lambda)}{\lambda} - \frac{\E_{\lambda}[X_n]}{\lambda n}\bigg]
		~=~ 0.
		\end{equation}
\end{enumerate}
Notice that \eqref{eq:2nd step} together with $\vel(0)=0$ implies
\begin{equation}	\label{eq:formula for the speed}
\vel'(0) = \lim_{\alpha\to\infty} \lim_{\substack{\lambda \to 0,\\ \lambda^{2} n \to \alpha}}
\frac{\E_{\lambda}[X_n]}{\lambda n} = \E[B(1) M(1)].
\end{equation}

\section{Background on the percolation model}	\label{sec:background}

In this section we provide some basic results on the percolation model.

\paragraph{Ergodicity of the percolation distribution.}
To ease notation, we identify $V$ with the additive group $\Z \times \Z_2$.
For instance, we write $(k,1)+(n,1)= (k+n,0)$ for $k,n \in \Z$.
With this notation, for $v \in V$, we define the shift $\theta^{v}:V \to V$, $w \mapsto w-v$.
The shift $\theta^{v}$ canonically extends to a mapping on the edges
and hence to a mapping on the configurations $\omega \in \Omega$.
In slight abuse of notation, we denote all these mappings by $\theta^{v}$.
The mappings $\theta^{v}$ form a commutative group since $\theta^{v} \theta^{w} = \theta^{v+w}$. 

The next result is contained in the proof of Lemma 5.5 in \cite{Axelson-Fisk+H"aggstr"om:2009b}.

\begin{lemma}	\label{Lem:omega is ergodic}
The probability measure $\Prmp^*$ is ergodic w.\,r.\,t.~the family of shifts $\theta^v$, $v \in V$,
that is, it is invariant under all shifts $\theta^v$ and for all shift-invariant sets $A \in \F$, we have $\Prmp^*(A) \in \{0,1\}$.
\end{lemma}

\paragraph{Cyclic decomposition.}
We introduce a decomposition of the percolation cluster into independent cycles.
A similar decomposition for the given model was first introduced in \cite{Axelson-Fisk+H"aggstr"om:2009b}.
If $(i,1)$ is isolated in $\omega$, we call $(i,0)$ a pre-regeneration point.
Cycles begin and end at pre-regeneration points.
These are bottlenecks in the graph which the walk has to visit in order to get past.
Let $\ldots, R^{\mathrm{pre}}_{-2}, R^{\mathrm{pre}}_{-1}, R^{\mathrm{pre}}_0, R^{\mathrm{pre}}_1, R^{\mathrm{pre}}_2, \ldots$
be an enumeration of the pre-regeneration points such that
$\ldots < \x(R^{\mathrm{pre}}_{-1}) < 0 \leq \x(R^{\mathrm{pre}}_0) <  \x(R^{\mathrm{pre}}_1) < \ldots$\,.
\medskip
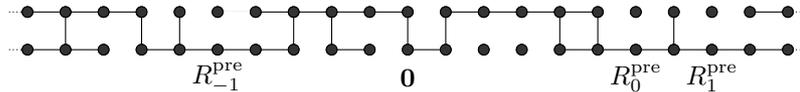
\begin{center}
\pgfmathsetseed{845133}
\begin{tikzpicture}[thin, scale=0.5,-,
                   shorten >=0pt+0.5*\pgflinewidth,
                   shorten <=0pt+0.5*\pgflinewidth,
                   every node/.style={circle,
                                      draw,
                                      fill          = black!80,
                                      inner sep     = 0pt,
                                      minimum width =4 pt}]

\def \p {0.5}

\foreach \x in {-10,-9,-8,-7,-6,-5,-4,-3,-2,-1,0,1,2,3,4,5,6,7,8,9,10}
\foreach \y in {0,1}
    \node at (\x,\y) {};

\foreach \x in {-10,-9,-8,-7,-6,-5,-4,-3,-2,-1,0,1,2,3,4,5,6,7,8,9}{
\foreach \y in {0,1}{
    \pgfmathparse{rnd}
    \let\dummynum=\pgfmathresult
    \ifdim\pgfmathresult pt < \p pt\relax \draw (\x,\y) -- (\x+1,\y);\fi
  }}
  
\foreach \x in {-10,-9,-8,-7,-6,-5,-4,-3,-2,-1,0,1,2,3,4,5,6,7,8,9,10}{
\foreach \y in {0}{
    \pgfmathparse{rnd}
    \let\dummynum=\pgfmathresult
    \ifdim\pgfmathresult pt < \p pt\relax \draw (\x,\y) -- (\x,\y+1);\fi
  }}

	\draw[densely dotted] (-10.5,1) -- (-10,1);
	\draw[densely dotted] (-10.5,0) -- (-10,0);
	\draw (-8,1) -- (-7,1);
	\draw (-3,1) -- (-2,1);
	\draw (1,1) -- (2,1);
	\draw (6,0) -- (7,0);
	\draw (9,0) -- (10,0);
	\draw[densely dotted] (10,0) -- (10.5,0);
	\draw[densely dotted] (10,1) -- (10.5,1);	
	\node[draw=none,fill=none] at (0,-0.75) {$\mathbf{0}$};
	\node[draw=none,fill=none] at (6,-0.75) {$R^{\mathrm{pre}}_{0}$};
	\node[draw=none,fill=none] at (8,-0.75) {$R^{\mathrm{pre}}_{1}$};

	\draw (9,1) -- (10,1);
	\draw[white, thick] (-5,1) -- (-4,1);
	\foreach \x in {-10,-9,-8,-7,-6,-5,-4,-3,-2,-1,0,1,2,3,4,5,6,7,8,9,10}
\foreach \y in {0,1}
    \node at (\x,\y) {};
    \node[draw=none,fill=none] at (-5,-0.75) {$R^{\mathrm{pre}}_{-1}$};
\end{tikzpicture}
\end{center}
Let $\Rpre$ be the set of all pre-regeneration points.
Let $E^{i,\leq}, E^{i, \geq} \subseteq E$
consist of those edges
with both endpoints having $\x$-coordinate $\leq i$ or $\geq i$, respectively.
Further, let $E^{i,<} \defeq E \setminus E^{i, \geq}$ and $E^{i,>} \defeq E \setminus E^{i,\leq}$.
We denote the subgraph of $\omega$ with vertex set $\{v \in V: a \leq \x(v) \leq b\}$
and edge set $\{e \in E^{a,\geq} \cap E^{b,<}: \omega(e)=1\}$
by $[a,b)$ and call $[a,b)$ a \emph{block} (of $\omega$).
The pre-regeneration points split the percolation cluster into blocks
\begin{center}
$\omega_n \defeq [\x(R_{n-1}^{\mathrm{pre}}), \x(R_n^{\mathrm{pre}})),	\quad	n \in \Z.$
\end{center}
There are infinitely many pre-regeneration points on both sides of the origin $\Prmp$-a.\,s.
The random walk $(Y_n)_{n \geq 0}$ under $\Prob$
can be viewed as a random walk among random conductances $(\omega(e))_{e \in E}$ (with additional self-loops).
For $n \in \Z$, we define $C_{n}$ to be the effective conductance between $R_{n-1}^{\mathrm{pre}}$
and $R_{n}^{\mathrm{pre}}$.
To be more precise, consider the $n$th cycle $\omega_n$ as a finite network.
Then the effective resistance between $R_{n-1}^{\mathrm{pre}}$
and $R_{n}^{\mathrm{pre}}$ is well-defined, see \cite[Section 9.4]{Levin+Peres+Wilmer:2009}.
We denote this effective resistance by $1/C_n$ and the effective conductance by $C_n$.
We further define $L_n \defeq L(\omega_n)$ to be the length of the $n$th cycle,
i.\,e., $L_n = \x(R_{n}^{\mathrm{pre}})-\x(R_{n-1}^{\mathrm{pre}})$.
We summarize the two definitions:
\begin{equation}	\label{eq:C_n and L_n}
C_n \defeq \mathcal{C}_{\text{eff}}(R_{n-1}^{\mathrm{pre}} \leftrightarrow R_{n}^{\mathrm{pre}})
\quad	\text{and}	\quad
L_n = \x(R_{n}^{\mathrm{pre}})-\x(R_{n-1}^{\mathrm{pre}}).
\end{equation}
For later use, we note the following lemma.

\begin{lemma}	\label{Lem:effective resistances}
The family $\{(C_{n},L_{n}): n \in \Z\}$ is independent and the $(C_{n},L_{n})$, $n \in \Z \setminus \{0\}$,
are identically distributed. Further, there is some $\vartheta > 0$ such that
\begin{equation}	\label{eq:exponential moments resistances}	\textstyle
\Ermp[\exp(\vartheta/C_{n})] + \Ermp[\exp(\vartheta L_n)] < \infty
\end{equation}
for all $n \in \Z$.
\end{lemma}
\begin{proof}
By Lemma 3.3 in \cite{Gantert+al:2018}, under $\Prmp$,
the family $(\theta^{R_{n-1}^{\mathrm{pre}}} \omega_n)_{n \in \Z}$
is independent and all cycles except cycle $0$ have the same distribution.
Hence the family $((C_n,L_n))_{n \in \Z}$ is independent and the $(C_n,L_n)$, $n \in \Z$, $n \not = 0$
are identically distributed.
Lemma 3.3(b) in \cite{Gantert+al:2018} gives that
$L_1 = \x(R_{1}^{\mathrm{pre}})-\x(R_0^{\mathrm{pre}})$ has a finite exponential moment of some order $\vartheta' > 0$.
By Raleigh's monotonicity law \cite[Theorem 9.12]{Levin+Peres+Wilmer:2009}, $C_{1}^{-1}$,
the effective resistance between $R_{1}^{\mathrm{pre}}$ and $R_0^{\mathrm{pre}}$,
increases if open edges between these two points are closed.
So, the effective resistance between $R_{1}^{\mathrm{pre}}$ and $R_0^{\mathrm{pre}}$
is bounded above by the effective resistance of the longest self-avoiding open path
connecting these two points.
This path has length at most $2L_1$
and thus, by the series law, resistance of at most $2L_1$.
Therefore, $C_{1}^{-1}$ has a finite exponential moment of order $\vartheta'/2$.

The proof of the statements concerning the cycle $\theta^{R_{-1}^{\mathrm{pre}}} \omega_0$
can be accomplished analogously, but requires revisiting the proof of Lemma 3.3 in \cite{Gantert+al:2018}.
We omit further details.
\end{proof}

We close this section with the definition of the backbone. We call a vertex $v$ forwards-communicating (in $\omega$)
if it is connected to $+\infty$ via an infinite open path that does not visit any vertex $u$ with $\x(u)<\x(v)$.
Finally, we define $\back \defeq \back(\omega) \defeq \{v \in V: v \text{ is forwards-communicating (in } \omega \text{)}\}$.

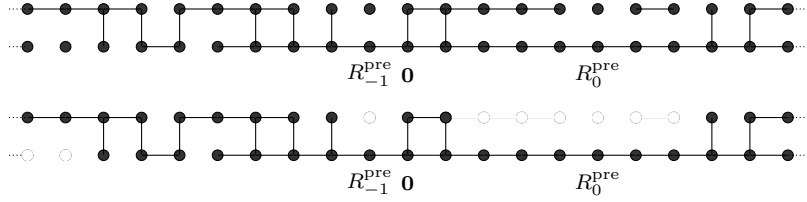
\begin{figure}[h]
\begin{center}
\pgfmathsetseed{216642}
\begin{tikzpicture}[thin, scale=0.5,-,
                   shorten >=0pt+0.5*\pgflinewidth,
                   shorten <=0pt+0.5*\pgflinewidth,
                   every node/.style={circle,
                                      draw,
                                      fill          = black!80,
                                      inner sep     = 0pt,
                                      minimum width =4 pt}]

\def \p {0.5}

\foreach \x in {-10,-9,-8,-7,-6,-5,-4,-3,-2,-1,0,1,2,3,4,5,6,7,8,9,10}
\foreach \y in {0,1}
    \node at (\x,\y) {};

\foreach \x in {-10,-9,-8,-7,-6,-5,-4,-3,-2,-1,0,1,2,3,4,5,6,7,8,9}{
\foreach \y in {1,0}{
    \pgfmathparse{rnd}
    \let\dummynum=\pgfmathresult
    \ifdim\pgfmathresult pt < \p pt\relax \draw (\x,\y) -- (\x+1,\y);\fi
  }}

\foreach \x in {-10,-9,-8,-7,-6,-5,-4,-3,-2,-1,0,1,2,3,4,5,6,7,8,9,10}{
\foreach \y in {0}{
    \pgfmathparse{rnd}
    \let\dummynum=\pgfmathresult
    \ifdim\pgfmathresult pt < \p pt\relax \draw (\x,\y) -- (\x,\y+1);\fi
  }}

	\draw	(-9,1) -- (-8,1);
	\draw	(-7,0) -- (-6,0);
	\draw	(-1,0) -- (0,0);
	\draw	(3,0) -- (4,0);
	\draw	(8,0) -- (9,0);

	\node[draw=none,fill=none] at (-1,-0.75) {$R^{\mathrm{pre}}_{-1}$};
	\node[draw=none,fill=none] at (0,-0.75) {$\mathbf{0}$};
	\node[draw=none,fill=none] at (5,-0.75) {$R^{\mathrm{pre}}_{0}$};

	\draw[densely dotted] (-10.5,0) -- (-10,0);
	\draw[densely dotted] (-10.5,1) -- (-10,1);
	\draw[densely dotted] (10,0) -- (10.5,0);
	\draw[densely dotted] (10,1) -- (10.5,1);
\end{tikzpicture}\smallskip

\pgfmathsetseed{216642}
\begin{tikzpicture}[thin, scale=0.5,-,
                   shorten >=0pt+0.5*\pgflinewidth,
                   shorten <=0pt+0.5*\pgflinewidth,
                   every node/.style={circle,
                                      draw,
                                      fill          = black!80,
                                      inner sep     = 0pt,
                                      minimum width =4 pt}]

\def \p {0.5}

\foreach \x in {-10,-9,-8,-7,-6,-5,-4,-3,-2,-1,0,1,2,3,4,5,6,7,8,9,10}
\foreach \y in {0,1}
    \node at (\x,\y) {};

\foreach \x in {-10,-9,-8,-7,-6,-5,-4,-3,-2,-1,0,1,2,3,4,5,6,7,8,9}{
\foreach \y in {1,0}{
    \pgfmathparse{rnd}
    \let\dummynum=\pgfmathresult
    \ifdim\pgfmathresult pt < \p pt\relax \draw (\x,\y) -- (\x+1,\y);\fi
  }}

\foreach \x in {-10,-9,-8,-7,-6,-5,-4,-3,-2,-1,0,1,2,3,4,5,6,7,8,9,10}{
\foreach \y in {0}{
    \pgfmathparse{rnd}
    \let\dummynum=\pgfmathresult
    \ifdim\pgfmathresult pt < \p pt\relax \draw (\x,\y) -- (\x,\y+1);\fi
  }}

	\draw	(-9,1) -- (-8,1);
	\draw	(-7,0) -- (-6,0);
	\draw	(-1,0) -- (0,0);
	\draw	(3,0) -- (4,0);
	\draw	(8,0) -- (9,0);

	\node[draw=none,fill=none] at (-1,-0.75) {$R^{\mathrm{pre}}_{-1}$};
	\node[draw=none,fill=none] at (0,-0.75) {$\mathbf{0}$};
	\node[draw=none,fill=none] at (5,-0.75) {$R^{\mathrm{pre}}_{0}$};

	\draw[densely dotted] (-10.5,0) -- (-10,0);
	\draw[densely dotted] (-10.5,1) -- (-10,1);
	\draw[densely dotted] (10,0) -- (10.5,0);
	\draw[densely dotted] (10,1) -- (10.5,1);

	\draw[white]	(1,1) -- (4,1);
	\draw[white]	(6,1) -- (7,1);
	\node at (1,1) {};
	\node[white] at (-9,0) {};
	\node[white] at (-10,0) {};

\foreach \x in {-1,2,3,4,5,6,7}
\foreach \y in {1}
    \node[white] at (\x,\y) {};
\end{tikzpicture}
\end{center}
\caption{The original percolation configuration and the backbone}
\label{Fig:backbone}
\end{figure}

\section{The environment seen from the walker and input from ergodic theory}

We define the \emph{process  of the environment seen from the particle}
$(\overline{\omega}(n))_{n \in \N_0}$ by $\overline{\omega}(n) \defeq \theta^{Y_n} \omega$, $n \in \N_0$.
It can be understood as a process under $P_{\omega}$ as well as under $\Prob_0$.
For later use, we shall show that $(\overline{\omega}(n))_{n \geq 0}$ is a reversible, ergodic Markov chain under $\Prob_0$.

\begin{lemma}	\label{Lem:environment seen from the particle}
The sequence $(\overline{\omega}(n))_{n \in \N_0}$ is a Markov process with state space $\Omega$
under $P_{\omega}$, $\Prob^*_0$ and $\Prob_0$, with initial distributions
$\delta_{\omega}$, $\Prmp^*$ and $\Prmp$, respectively.
In each of these cases, the transition kernel $M(\omega,\domega')$ is given by
\begin{equation}	\label{eq:M}
M(\omega,\!f)
=	E_{\omega}[f(\theta^{Y_1} \omega)]
=	\frac13 \!
\sum_{v \sim \mathbf{0}} \!\big(\1_{\{\omega(\mathbf{0},v)=1\}} f(\theta^v \omega) + \1_{\{\omega(\mathbf{0},v)=0\}} f(\omega)\big),
\end{equation}
$\omega \in \Omega$, $f$ nonnegative and $\F$-measurable.
Moreover,
$(\overline{\omega}(n))_{n \in \N_0}$ is reversible and ergodic under $\Prob_0$.
\end{lemma}
\begin{proof}
The proof of \eqref{eq:M} is a standard calculation and can be done along the lines of the corresponding one for random walk in random environment,
see e.\,g.\ \cite[Lemma 2.1.18]{Zeitouni:2004}.
To prove reversibility of $(\overline{\omega}(n))_{n \in \N_0}$ under $\Prob_0$, notice that,
for every bounded and measurable $f: \Omega \to \R$ and all $v \in V$ with $v \sim \bfnull$,
\begin{equation}	\label{eq:shift invariance}
E_p [f(\theta^v \omega) \1_{\{\omega(\mathbf{0},v)=1\}}]
~=~	E_p[f(\omega) \1_{\{\omega(\mathbf{0},-v)=1\}}].
\end{equation}
This can be verified along the lines of the proof of (2.2) in \cite{Berger+Biskup:2007}.
Arguing as in the proof of Lemma 2.1 in \cite{Berger+Biskup:2007},
\eqref{eq:shift invariance} implies that
\begin{equation}	\label{eq:reversibility}
E_p [f(\omega) M\!g(\omega)]	~=~	E_p [g(\omega) M\!f(\omega)]
\end{equation}
for all measurable and bounded $f,g: \Omega \to \R$,
which is the reversibility of $(\overline{\omega}(n))_{n \in \N_0}$ under $\Prob_0$.
To prove ergodicity, we argue as in the proof of Lemma 4.3 in \cite{DeMasi+al:1989}.
Fix an invariant set $A' \subseteq \Omega$, i.\,e.,
\begin{equation*}
M(\omega,A')	~=~	1	\quad	\text{for } \Prmp\text{-almost all } \omega \in A'.
\end{equation*}
By Corollary 5 on p.\;97 in \cite{Rosenblatt:1971},
it is enough to show that $\Prmp(A') \in \{0,1\}$.
If $\Prmp(A') = 0$, there is nothing to show.
Thus assume that $\Prmp(A') > 0$.
Since $\Prmp$ is concentrated on $\Omega_{\mathbf{0}}$,
we can ignore the part of $A'$ that is outside $\Omega_{\mathbf{0}}$
and can thus assume $A' \subseteq \Omega_{\mathbf{0}}$.
From the form of $M$, we deduce that
\begin{equation*}
\text{for } \Prmp \text{-almost all } \omega \in A' \text{ we have that } \theta^v \omega \in A' \text{ for all } v \in \Cluster(\omega,\mathbf{0}).
\end{equation*}
To avoid trouble with exceptional sets of $\Prmp$-probability $0$,
we define $A \defeq \{\omega \in A': \theta^v \omega \in A' \text{ for all } v \in \Cluster(\omega,\mathbf{0})\}$.
Since $A \subseteq A'$, it suffices to show $\Prmp(A)=1$.
First notice that $\Prmp(A) = \Prmp(A') > 0$ and that
\begin{equation}	\label{eq:translation invariance of A}
\omega \in A \text{ and }  v \in \Cluster(\omega,\mathbf{0}) \text{ imply }\theta^v \omega \in A.
\end{equation}
Plainly,
\begin{equation*}
A \subseteq B \defeq \{\omega \in \Omega: \theta^v \omega \in A \text{ for some } v \in V\}.
\end{equation*}
By definition, $B$ is invariant under shifts $\theta^v$, $v \in V$.
Since $\Prmp^*(B) \geq \Prmp^*(A) > 0$, the ergodicity of $\Prmp^*$
(see Lemma \ref{Lem:omega is ergodic}) implies $\Prmp^*(B) = 1$.
We shall now show that $B \cap \Omega_{\mathbf{0}} \subseteq A$ up to a set of $\Prmp^*$ measure zero.
Once this is shown, we can conclude that $\Prmp^*(A) = \Prmp^*(\Omega_{\mathbf{0}})$,
in particular, $\Prmp(A)=1$.
In order to show $B \cap \Omega_{\mathbf{0}} \subseteq A$ $\Prmp^*$-a.\,s.,
pick an arbitrary $\omega \in B \cap \Omega_{\mathbf{0}}$ such that $\omega$ has only one infinite connected component $\Cluster_{\infty}$.
(The set of $\omega$ with this property has measure $1$ under $\Prmp^*$.)
By definition of the set $B$, there exists a $v \in V$ such that $\theta^v \omega \in A$.
Since $A \subseteq \Omega_{\mathbf{0}}$, the origin $\mathbf{0}$ must be in the infinite connected component of $\theta^v \omega$
or, equivalently, $v$ is in the infinite connected component of $\omega$.
From the uniqueness of the infinite connected component together with $\omega \in \Omega_{\mathbf{0}}$,
we thus infer $v \in \Cluster(\omega,\mathbf{0})$.
This is equivalent to $-v \in \Cluster(\theta^v \omega, \mathbf{0})$.
Together with $\theta^v \omega \in A$ this implies $\omega = \theta^{-v} \theta^v \omega \in A$
by means of \eqref{eq:translation invariance of A}.
The proof is complete.
\end{proof}

The lemma has the following useful corollary.

\begin{lemma}	\label{Lem:Birkhoff quenched}
Let $f: \Omega \to \R$ be integrable with respect to $\Prmp$. Then, for $\Prmp$-almost all $\omega \in \Omega$,
we have
\begin{equation}	\label{eq:Birkhoff quenched}
\lim_{n \to \infty} \frac1n \sum_{k=0}^{n-1} f(\overline{\omega}(k)) = \Ermp[f]	\quad	P_{\omega}\text{-a.\,s.}
\end{equation}
\end{lemma}
\begin{proof}
As $(\overline{\omega}(n))_{n \in \N_0}$ is reversible and ergodic with respect to $\Prob_0$,
we infer
\begin{equation}	\label{eq:Birkhoff annealed}
\lim_{n \to \infty} \frac1n \sum_{k=0}^{n-1} f(\overline{\omega}(k)) = \E_0[f(\overline{\omega}(0))]	\quad	\Prob_0 \text{-a.\,s.}
\end{equation}
from Birkhoff's ergodic theorem.
As the law of $\overline{\omega}(0)$ under $\Prob_0$ is $\Prmp$, we have $\E_0[f(\overline{\omega}(0))] = \Ermp[f]$.
Hence \eqref{eq:Birkhoff quenched} follows from \eqref{eq:Birkhoff annealed} and the definition of $\Prob_0$.
\end{proof}

Next we see that if the walker sees the same environment at two different epochs,
then, with probability $1$, the position of the walker at those two epochs is actually the same.
This allows to reconstruct the random walk from the environment seen from the walker.

\begin{lemma}	\label{Lem:y_n<->omegabar(n)}
We have $\Prmp^*(\theta^v \omega = \theta^{w} \omega) = 0$
for all $v, w \in V$, $v \not = w$.
The same statement holds with $\Prmp^*$ replaced by $\Prmp$.
\end{lemma}
\begin{proof}
By shift invariance, we may assume $w=\mathbf{0}$ and $v \not= \mathbf{0}$.
Call a vertex $u \in V$ backwards-communicating (in $\omega$)
if there exists an infinite open path in $\omega$ which contains $u$ but no vertex
with strictly larger $\x$-coordinate.
Define ${\tt T} = ({\tt T}_i)_{i \in \Z}$ by letting
\begin{equation*}
{\tt T}_i	~\defeq~	\begin{cases}
			{\tt 00}	&	\text{if neither $(i,0)$ nor $(i,1)$ are backwards-communicating;}	\\
			{\tt 01}	&	\text{if $(i,0)$ is not backwards-communicating but $(i,1)$ is;}	\\
			{\tt 10}	&	\text{if $(i,0)$ is backwards-communicating but $(i,1)$ is not;}	\\
			{\tt 11}	&	\text{if $(i,0)$ and $(i,1)$ are backwards-communicating.}
			\end{cases}
\end{equation*}
Notice that $\theta^v \omega = \omega$ implies $\theta^v {\tt T} = {\tt T}$ where $\theta^v {\tt T}_i$ records
whether the vertices $(i,0)$ and $(i,1)$ are backwards-communicating in $\theta^v \omega$.
Under $\Prmp^*$, ${\tt T}$ is a time-homogeneous, stationary, irreducible and aperiodic Markov chain
with state space $\{{\tt 01}, {\tt 10}, {\tt 11}\}$
by \cite[Theorem 3.1 and the form of the transition matrix on p.\;1111]{Axelson-Fisk+H"aggstr"om:2009}.
From this one can deduce $\Prmp^*(\theta^v {\tt T} = {\tt T}) = 0$ and, in particular, $\Prmp^*(\theta^v \omega = \omega) = 0$.
Finally, notice that, for every event $A \in \F$, $\Prmp^*(A) = \Prmp(A)$ whenever $\Prmp^*(A) \in \{0,1\}$.
\end{proof}

\begin{lemma}	\label{Lem:ergodic theory input}
The increment sequences $(X_n-X_{n-1})_{n \in \N}$ and $(Y_n-Y_{n-1})_{n \in \N}$ are ergodic under $\Prob$.
\end{lemma}
\begin{proof}
Define a mapping $\varphi: \Omega \times \Omega \to V$ via
\begin{equation*}
\varphi(\omega,\omega') =
\begin{cases}
v			&	\text{if }	\omega' = \theta^v \omega \text{ for a unique } v \in V;	\\
\mathbf{0}		&	\text{otherwise.}
\end{cases}
\end{equation*}
It can be checked that $\varphi$ is product-measurable.
Further, $\Prob$-a.\,s.,
$Y_n - Y_{n-1} = \varphi(\overline{\omega}(n\!-\!1),\overline{\omega}(n))$
for all $n \in \N$.
Combining the ergodicity of $(\overline{\omega}(n))_{n \in \N_0}$
under~$\Prob_0$ (see Lemma \ref{Lem:environment seen from the particle})
with Lemma 5.6(c) in \cite{Axelson-Fisk+H"aggstr"om:2009},
we infer that $(Y_n-Y_{n-1})_{n \in \N} = (\varphi(\overline{\omega}(n\!-\!1),\overline{\omega}(n)))_{n \in \N}$
is ergodic.
(To formally apply the lemma, one may extend $(\overline{\omega}(n))_{n \in \N_0}$
to a stationary ergodic sequence $(\overline{\omega}(n))_{n \in \Z}$ using reversibility.)
Then also $(X_n-X_{n-1})_{n \in \N} = (\x(Y_n-Y_{n-1}))_{n \in \N}$ is ergodic under $\Prob$
again by Lemma 5.6(c) in \cite{Axelson-Fisk+H"aggstr"om:2009}.
\end{proof}

\section{Preliminary results}	\label{sec:preliminaries}

\paragraph{Hitting probabilities.}

The next lemma provides bounds on hitting probabilities for biased random walk
that we use later on.

\begin{lemma}	\label{Lem:hitting probabilities}
Let $L,R \in \N$ and $u,v,w \in \Rpre$ be such that $\x(v)-\x(u)=L\lfloor \frac1\lambda \rfloor$
and $\x(w)-\x(v)=R \lfloor \frac1\lambda \rfloor$. Then,
with $T_u$ and $T_w$ denoting the first hitting times of $(Y_n)_{n \in \N_0}$ at $u$ and $w$, respectively,
we have
\begin{equation}	\label{eq:comparewalks:1}
\frac{1-e^{-R}}{1-e^{-R}+  6 (e^{2L}-1)} \leq P_{\omega,\lambda}^{v}(T_{u}< T_{w}) \leq  \frac{1-e^{-2R}}{1-e^{-2R}+ \frac15 (e^{L}-1)}
\end{equation}
for all sufficiently small $\lambda \in (0,\lambda_0]$ for some $\lambda_0 > 0$ not depending on $L,R$.
In particular, for these $\lambda$,
\begin{equation}	\label{eq:comparewalks:2}
\frac1{6e^{2L}-5}\leq P_{\omega,\lambda}^{v}(T_{u} < \infty) \leq \frac{5}{4 + e^{L}}.
\end{equation}
Moreover, for $L=1$ and $R=\infty$, for all sufficiently small $\lambda > 0$, we have
\begin{equation}	\label{eq:comparewalks:3}
P_{\omega,\lambda}^{v}(T_{u} < \infty) \leq \frac4{10}.
\end{equation}
\end{lemma}
\begin{proof}
Since $u,w$ are pre-regeneration points,
it suffices to consider the finite subgraph $[u,w)$.
As $v$ is also a pre-regeneration point,
standard electrical network theory gives
\begin{equation*}
P_{\omega,\lambda}^{v}(T_{u}< T_{w}) = \frac{\Res(v \leftrightarrow w)}{\Res(u \leftrightarrow v)+\Res(v \leftrightarrow w)}
\end{equation*}
where $\Res(u \leftrightarrow v) = \mathcal{R}_{\text{eff},\lambda}(u \leftrightarrow v)$ denotes the effective resistance between $u$ and $v$ in $[u,w)$,
see \cite[Section 9.4]{Levin+Peres+Wilmer:2009}.
Let us first prove the upper bound by applying Raleigh's monotonicity law \cite[Theorem 9.12]{Levin+Peres+Wilmer:2009}.
We add all edges left of $v$ that were not present in the cluster $\omega$.
This decreases the effective resistance $\Res(u \leftrightarrow v)$ between $u$ and $v$.
On the right of $v$ we delete all edges but a simple path from $v$ to $w$.
This increases the effective resistance $\Res(v \leftrightarrow w)$. 
\begin{figure}[htp]
\begin{center}
\begin{minipage}{10cm}
\begin{center}
\begin{tikzpicture}[thin, scale=0.5,-,
                   shorten >=0pt+0.5*\pgflinewidth,
                   shorten <=0pt+0.5*\pgflinewidth,
                   every node/.style={circle,
                                      draw,
                                      fill          = black!80,
                                      inner sep     = 0pt,
                                      minimum width =4 pt}]
\path[draw] 
       node at (0,0) {}  
       node at (0,1) {} 
       node at (1,0) {} 
       node at (1,1) {} 
       node at (2,0) {} 
       node at (2,1) {} 
       node at (3,0) {} 
       node at (3,1) {}
       node at (4,0) {} 
       node at (4,1) {}
       node at (5,0) {} 
       node at (5,1) {} 
       node at (-1,0){}
       node at (-1,1) {} 
       node at (-2,0) {} 
       node at (-2,1) {} 
       node at (-3,0) {} 
       node at (-3,1) {} 
        ; 

	\draw (-3,0) -- (-2,0) ;
	\draw (-2,0) -- (-2,1) ;
	\draw (-2,1) -- (-1,1) ;
	\draw (-1,1) -- (-1,0) ;
	\draw (-1,0) -- (0,0) ;
	\draw (0,1) -- (0,0) ;
	\draw (0,0) -- (2,0) ;
	\draw (2,0) -- (3,0) ;
	\draw (2,0) -- (2,1) ;
	\draw (2,0) -- (3,0) ;
	\draw (0,0) -- (3,0) ;
	\draw (2,1) -- (4,1) ;
	\draw (4,0) -- (5,0) ;
	\draw (4,0) -- (4,1)  ;


\node[draw=none,fill=none] at (-3,-0.4) {$u$};
\node[draw=none,fill=none] at (1,-0.4) {$v$};
\node[draw=none,fill=none] at (5,-0.4) {$w$};
\node[draw=none,fill=none,right] at (6,0.5) {$[u,w)$};
\end{tikzpicture}
\end{center}
\end{minipage}
\hfill
\begin{minipage}{10cm}
\begin{center}
\begin{tikzpicture}[thin, scale=0.5,-,
                   shorten >=0pt+0.5*\pgflinewidth,
                   shorten <=0pt+0.5*\pgflinewidth,
                   every node/.style={circle,
                                      draw,
                                      fill          = black!80,
                                      inner sep     = 0pt,
                                      minimum width =4 pt}]
\path[draw] 
       node at (0,0) {}  
       node at (0,1) {} 
       node at (1,0) {} 
       node at (1,1) {} 
       node at (2,0) {} 
       node at (2,1) {} 
       node at (3,0) {} 
       node at (3,1) {}
       node at (4,0) {} 
       node at (4,1) {}
       node at (5,0) {} 
       node at (5,1) {} 
       node at (-1,0){}
       node at (-1,1) {} 
       node at (-2,0) {} 
       node at (-2,1) {} 
       node at (-3,0) {} 
       node at (-3,1) {} 
        ; 

	\draw (-3,0) -- (-3,1) ;
	\draw (-3,0) -- (1,0) ;
	\draw (-3,1) -- (1,1) ;
	\draw (-2,0) -- (-2,1) ;
	\draw (-2,1) -- (-1,1) ;
	\draw (-1,1) -- (-1,0) ;
	\draw (-1,0) -- (0,0) ;
	\draw (0,1) -- (0,0) ;
	\draw (0,1) -- (1,1) ;
	\draw (1,0) -- (1,1) ;
	\draw (0,0) -- (2,0) ;
	\draw (2,0) -- (2,1) ;
	\draw (2,1) -- (4,1) ;
	\draw (4,0) -- (5,0) ;
	\draw (4,0) -- (4,1)  ;


\node[draw=none,fill=none] at (-3,-0.4) {$u$};
\node[draw=none,fill=none] at (1,-0.4) {$v$};
\node[draw=none,fill=none] at (5,-0.4) {$w$};
\node[draw=none,fill=none,right] at (6,0.5) {$\hphantom{[u,w)}$};
\node[draw=none,fill=none,right] at (6,0.5) {$G$};
\end{tikzpicture}
\end{center}
\end{minipage}
\end{center}
\caption{Construction of the graph $G$.}
\label{fig:graphG}
\end{figure}
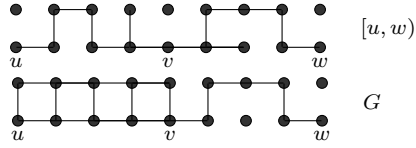
The graph obtained is denoted by $G$, see Figure \ref{fig:graphG} for an example.
We conclude
\begin{equation*}
P_{\omega,\lambda}^{v}(T_{u}< T_{w}) \leq
\frac{\mathcal{R}_{\mathrm{eff},G}(v \leftrightarrow w)}{\mathcal{R}_{\mathrm{eff},G}(u \leftrightarrow v)+\mathcal{R}_{\mathrm{eff},G}(v \leftrightarrow w)}
\end{equation*}
where $\mathcal{R}_{\mathrm{eff},G}$ denote the corresponding effective resistances in $G$.
We may assume without loss of generality that $v = \mathbf{0}$.
Then, by the series law, we can bound $\mathcal{R}_{\mathrm{eff},G}(v \leftrightarrow w)$ from above by
\begin{equation*}
\mathcal{R}_{\mathrm{eff},G}(v \leftrightarrow w) \leq {\textstyle \sum_{k=1}^{2R \lfloor \frac1\lambda \rfloor}} e^{-\lambda k}
= \frac{e^{-\lambda}}{1-e^{-\lambda}} (1-e^{-\lambda 2R \lfloor \frac1\lambda \rfloor})
\leq \frac{1-e^{-2R}}{1-e^{-\lambda}}.
\end{equation*}
From the Nash-Williams inequality \cite[Proposition 9.15]{Levin+Peres+Wilmer:2009}, we infer
\begin{equation*}
\mathcal{R}_{\mathrm{eff},G}(u \leftrightarrow v) \geq {\textstyle \sum_{k=1}^{L \lfloor \frac1\lambda \rfloor}} \big( 2 e^{-\lambda (2k-1)} \big)^{-1} 
= \frac{e^{\lambda}}2 \frac{e^{2\lambda L \lfloor \frac1\lambda \rfloor}-1}{e^{2\lambda}-1} 
\geq \frac{1}2 \frac{e^{2(1-\lambda)L}-1}{e^{2\lambda}-1}.
\end{equation*}
Consequently,
\begin{align}
P_{\omega,\lambda}^{v}(T_{u}< T_{w})
& \leq	\frac{1-e^{-2R}}{1-e^{-2R}+ \frac12 (e^{2(1-\lambda)L}-1) \frac{1-e^{-\lambda}}{e^{2\lambda}-1}}	\label{eq:uniform bound}	\\
&\leq		\frac{1-e^{-2R}}{1-e^{-2R}+ \frac15 (e^{L}-1)}	\notag
\end{align}
for all sufficiently small $\lambda > 0$ independent of $L,R$. 
The proof of the lower bound is similar.
We add all edges right of $v$ and keep only one simple path left of $v$.
For this new graph $\widetilde G$ we bound the effective resistances as follows. 
From the Nash-Williams inequality \cite[Proposition 9.15]{Levin+Peres+Wilmer:2009}, we infer
\begin{equation*}
\mathcal{R}_{\mathrm{eff},\widetilde G}(v \leftrightarrow w) \geq {\textstyle \sum_{k=1}^{R \lfloor \frac1\lambda \rfloor}} \big( 2 e^{\lambda (2k-1)} \big)^{-1}
= \frac{e^{-\lambda}}2 \frac{1-e^{-\lambda 2R \lfloor \frac1\lambda \rfloor}}{1-e^{-2\lambda}} 
\geq \frac15 \frac{1-e^{-R}}{1-e^{-\lambda}}
\end{equation*}
for all sufficiently small $\lambda > 0$. Moreover,
\begin{equation*}
\mathcal{R}_{\mathrm{eff},\widetilde G}(u \leftrightarrow v) \leq   {\textstyle \sum_{k=1}^{2L \lfloor \frac1\lambda \rfloor}} e^{\lambda k}
= \frac{e^{\lambda}}{e^{\lambda}-1} (e^{\lambda 2L \lfloor \frac1\lambda \rfloor}-1)
\leq \frac{e^{\lambda}}{e^{\lambda}-1} (e^{2L}-1).
\end{equation*}
The lower bound in \eqref{eq:comparewalks:1} now follows.
Equation \eqref{eq:comparewalks:2} follows from \eqref{eq:comparewalks:1} by letting $R \to \infty$.
Equation \eqref{eq:comparewalks:3} follows from \eqref{eq:uniform bound}
and the observation that the term on the right-hand side of \eqref{eq:uniform bound} with $R = \infty$ and $L=1$
tends to $\frac{4}{3+e^{2}}=0.3850\ldots$ for $\lambda \to 0$.
\end{proof}

\section{Harmonic functions and the corrector}	\label{sec:harmonic functions}

We use harmonic functions to construct a martingale approximation for $X_n$.
As a result, $X_n$ can be written as a martingale plus a corrector.

\paragraph{The corrector method.}
The corrector method has been used in various setups,
see e.g.\ \cite{Berger+Biskup:2007} and \cite{Mathieu+Piatnitski:2007}.
In the present setup, the method works as follows.

\noindent
We seek a function $\psi: \Omega \times V \to \R$ such that, for each fixed $\omega$,
$\psi(\omega,\cdot)$ is harmonic in the second argument with respect to
the transition kernel of $P_{\omega}$, that is,
$E^v_{\omega} [\psi(\omega,Y_1)] = \psi(\omega,v)$ for all $v \in V$.
In what follows, we shall sometimes suppress the dependence of $\psi$ on $\omega$ in the notation so that the above condition becomes
\begin{equation}	\label{eq:psi}
E^v_{\omega} [\psi(Y_1)] = \psi(v),	\quad	v \in V.
\end{equation}
If we find such a function, then $(\psi(Y_n))_{n \in \N_0}$ is a martingale under $P_{\omega}$.
We can then define $\chi(\omega,v) \defeq \x(v) - \psi(\omega,v)$ and get that
$X_n = \psi(Y_n) + \chi(Y_n)$. In other words,
$X_n$ can be written as the $n$th term in a martingale, $\psi(Y_n)$,
plus a \emph{corrector}, $\chi(Y_n)$.
In order to derive a central limit theorem for $X_n$,
it then suffices to apply the martingale central limit theorem to $\psi(Y_n)$
and to show that the contribution of $\chi(Y_n)$ is asymptotically negligible.

\paragraph{Construction of a harmonic function.}
Let $\omega \in \Omega_{\mathbf{0}}$ be such that there are infinitely many pre-regeneration points
to the left and to the right of the origin (the set of these $\omega$ has $\Prmp$-probability $1$).
Then, under $P_{\omega}$, the walk $(Y_n)_{n \geq 0}$
is the simple random walk on the unique infinite cluster $\Cluster_\omega$.
It can also be considered as a random walk among random conductances where
each edge $e \in E$ has conductance $\omega(e)$.
Recall that $C_{n}$ denotes the effective conductance between $R_{n-1}^{\mathrm{pre}}$
and $R_{n}^{\mathrm{pre}}$, see \eqref{eq:C_n and L_n}.
We couple our model with the random conductance model on $\Z$ with conductance $C_{n}$
between $n-1$ and $n$.
For the latter model, the harmonic functions are known.
In fact,
\begin{equation}	\label{eq:Psi harmonic RCM}
\Psi(n) \defeq
\begin{cases}
-\sum_{k=n+1}^{0} \frac{1}{C_{k}}			&	\text{for } n<0,	\\
0									&	\text{for } n=0,	\\
\sum_{k=1}^{n} \frac{1}{C_{k}}				&	\text{for } n>0,	\\
\end{cases}
\end{equation}
is harmonic for the random conductance model on $\Z$.
We define $\phi(R_{n}^{\mathrm{pre}}) \defeq \Psi(n)$.
Now fix an arbitrary $n \in \Z$.
For any vertex $v \in \omega_n$,
we then set $\phi(v) \defeq \phi(R_{n-1}^{\mathrm{pre}}) + \Res(R_{n-1}^{\mathrm{pre}} \leftrightarrow v)$,
where the latter expression denotes the effective resistance between $R_{n-1}^{\mathrm{pre}}$ and $v$ in the finite network $\omega_n$.
This definition is consistent with the cases $v = R_{n-1}^{\mathrm{pre}}, R_{n}^{\mathrm{pre}}$,
and, by \cite[Eq.\;(9.12)]{Levin+Peres+Wilmer:2009}, makes $\phi$ harmonic on $\omega_n$.
As $n \in \Z$ was arbitrary, $\phi$ is harmonic on $\Cluster_\omega$.
We now define
\begin{equation*}	\textstyle
\psi(\omega,v)	\defeq	\frac{\Ermp[L_1]}{\Ermp[C_{1}^{-1}]} (\phi(\omega,v) - \phi(\omega,\mathbf{0})),	\quad	v \in \Cluster_{\omega},
\end{equation*}
where we remind the reader that the $L_n$, $n \in \Z$ were defined in \eqref{eq:C_n and L_n}.
Notice that the expectations $\Ermp[L_1]$ and $\Ermp[C_{1}^{-1}]$ are finite by Lemma \ref{Lem:effective resistances}.
Since $\phi(\omega,\cdot)$ is harmonic under $P_{\omega}$, so is $\psi(\omega,\cdot)$
as an affine transformation of $\phi(\omega,\cdot)$.
It turns out that $\psi$ is more suitable for our purposes as $\psi$ is additive in a certain sense.
Next, we collect some facts about $\psi$.

\begin{proposition}	\label{Prop:corrector}
Consider the function $\psi: \Omega \times V \to \R$ constructed above.
The following assertions hold:
\begin{enumerate}[(a)]
	\item
		For $\Prmp$-almost all $\omega \in \Omega_{\mathbf{0}}$,
		the function $v \mapsto \psi(\omega,v)$ is harmonic with respect to (the transition probabilities of) $P_{\omega}$.
	\item
		For $\Prmp$-almost all $\omega \in \Omega_{\mathbf{0}}$ and all $u,v \in \Cluster_\omega$,
		it holds that
		\begin{equation}	\label{eq:additivity of psi}	\textstyle
		\psi(\omega,u+v) = \psi(\omega,u) + \psi(\theta^u \omega,v).
		\end{equation}
	\item
		For $\Prmp$-almost all $\omega \in \Omega_{\mathbf{0}}$
		\begin{equation}	\label{eq:uniform bound fo increments of psi}
		\sup_{v \sim w} |\psi(\omega,v)-\psi(\omega,w)| \1_{\{v \in \Cluster_\omega\}} \1_{\{\omega(\langle v,w \rangle)=1\}}\leq {\textstyle \frac{\Ermp[L_1]}{\Ermp[C_{1}^{-1}]}}.
		\end{equation}
\end{enumerate}
\end{proposition}
\begin{proof}
Assertion (a) is clear from the construction of $\psi$.
For the proof of (b), in order to ease notation,
we drop the factor $\Ermp[L_1]/\Ermp[C_{1}^{-1}]$ in the definition of $\psi$.
This is no problem as \eqref{eq:additivity of psi} remains true after multiplication by a constant.
Now fix $\omega \in \Omega_{\mathbf{0}}$ such that there are infinitely many pre-regeneration points
to the left and to the right of $\mathbf{0}$. The set of these $\omega$ has $\Prmp$-probability $1$.
Let $u,v \in \Cluster_{\omega}$. 
We suppose that there are $n,m \in \N_0$ such that $u \in \omega_n$ and $v \in \omega_{n+m}$.
The other cases can be treated similarly.
We further assume that $\y(u)=0$.
Define $T \defeq \inf\{n \in \N_0: Y_n \in \Rpre\}$ to be the first hitting time of the set of pre-regeneration points.
From Proposition 9.1 in \cite{Levin+Peres+Wilmer:2009}, we infer
\begin{align}
\psi(\omega,u)
&=
\sum_{k=1}^{n-1} \tfrac{1}{C_k(\omega)} + P_{\omega}^u(Y_T = R_n^{\mathrm{pre}}) \tfrac{1}{C_n(\omega)} - \phi(\omega,\mathbf{0})	\label{eq:psi(omega,u)}	\\
\text{and }
\psi(\theta^u\omega,v)
&=
\sum_{k=n+1}^{n+m-1} \tfrac{1}{C_k(\omega)} + P_{\theta^u \omega}^v(Y_T = R_{m}^{\mathrm{pre}}) \tfrac{1}{C_{n+m}(\omega)} - \phi(\theta^u \omega,\mathbf{0})	\label{eq:psi(theta^u omega,v)}
\end{align}
where we have used that $\y(u)=0$ which implies that the pre-regeneration points in $\theta^u \omega$
are the pre-regeneration points in $\omega$ but shifted by index $n$ as $u \in \omega_n$.
Here,
\begin{align*}
P_{\theta^u \omega}^v(Y_T = R_{m}^{\mathrm{pre}})
=	\textstyle
P_{\omega}^{u+v}(Y_T = R_{n+m}^{\mathrm{pre}})	\\
\text{and}	\quad	
- \phi(\theta^u \omega,\mathbf{0}) =	\textstyle
P_{\theta^u \omega}^{\mathbf{0}}(Y_T=R_{-1}^{\mathrm{pre}}) \frac{1}{C_{n}} = P_{\omega}^u (Y_T=R_{n-1}^{\mathrm{pre}}) \frac{1}{C_{n}}.
\end{align*}
Using the last two equations in \eqref{eq:psi(theta^u omega,v)}
and summing over \eqref{eq:psi(omega,u)} and \eqref{eq:psi(theta^u omega,v)} gives:
\begin{align*}	\textstyle
\psi(\omega,u) + \psi(\theta^u \omega,v)
&=	\textstyle
\sum_{k=1}^{n+m-1} \frac{1}{C_k} + P_{\omega}^{u+v}(Y_T = R_{n+m}^{\mathrm{pre}}) \frac{1}{C_{n+m}} - \phi(\omega,\mathbf{0}) \\
&=	\textstyle
\psi(\omega,u+v).
\end{align*}
The proof in the case $\y(u)=1$ is similar but requires more cumbersome calculations as the pre-regenerations change when considering $\theta^u \omega$
instead of $\omega$ due to the flip of the cluster. However, the pre-regeneration points $\omega$ remain pivotal edges in $\theta^u \omega$
and by the series law the corresponding resistances add. We refrain from providing further details.

\noindent
We now turn to the proof of assertion (c).
According to the definition of $\psi$, the statement is equivalent to
\begin{equation}	\label{eq:uniform bound fo increments of phi}
\sup_{v \sim w} |\phi(\omega,v)-\phi(\omega,w)| \1_{\{v \in \Cluster_\omega\}} \1_{\{\omega(\langle v,w \rangle)=1\}} \leq 1
\quad \text{ for $\Prmp$-almost all $\omega$}.
\end{equation}
For the proof of \eqref{eq:uniform bound fo increments of phi}, pick $\omega \in \Omega_{\mathbf{0}}$
such that there are infinitely many pre-regeneration points to the left and to the right of the origin.
Now pick $v,w \in \Cluster_\omega$ with $\omega(\langle v,w \rangle) = 1$.
Then there is an $n \in \Z$ such that $v,w$ are vertices of $\omega_n$ and $\langle v,w \rangle$ is an edge of $\omega_n$.
In this case, by the definition of $\phi$,
$\phi(v) \defeq \phi(a) + \Res(a \leftrightarrow v)$
and $\phi(w) \defeq \phi(a) + \Res(a \leftrightarrow w)$
for $a=R_{n-1}^{\mathrm{pre}}$
where $\Res(u \leftrightarrow u')$ denotes the effective resistance between $u$ and $u'$ in the finite network $\omega_n$.
To unburden notation, we assume without loss of generality that $\phi(a)=0$.
Then $\Res(\cdot \leftrightarrow \cdot)$ is a metric on the vertex set of $\omega_n$, see \cite[Exercise 9.8]{Levin+Peres+Wilmer:2009}.
In particular, $\Res(\cdot \leftrightarrow \cdot)$ satisfies the triangle inequality.
This gives
\begin{equation*}
\phi(w) = \Res(a\!\leftrightarrow\!w) \leq \Res(a\!\leftrightarrow\!v) + \Res(v\!\leftrightarrow\!w) \leq \Res(a\!\leftrightarrow\!v) + 1 = \phi(v)+1,
\end{equation*}
where we have used that $\Res(v \leftrightarrow w) \leq 1$.
This inequality follows from Raleigh's monotonicity principle \cite[Theorem 9.12]{Levin+Peres+Wilmer:2009}
when closing all edges in $\omega_n$ except $\langle v,w \rangle$.
By symmetry, we also get $\phi(v) \leq \phi(w)+1$ and, hence, $|\phi(v)-\phi(w)| \leq 1$.
\end{proof}

For $v \in V$ (and fixed $\omega$), we define $\chi(v) \defeq \x(v) - \psi(v)$.
Then $X_n = \chi(Y_n) + \psi(Y_n)$.
For the proof of the Einstein relation,
we require strong bounds on the corrector $\chi$.
These bounds are established in the following lemma.

\begin{lemma}	\label{Lem:almost sure bounds corrector}
For any $\varepsilon \in (0,\frac12)$ and every sufficiently small $\delta > 0$ there is a random variable $K$ on $\Omega$ with
$\Ermp[K^2]<\infty$ such that
\begin{equation}	\label{eq:pathwise bound on chi(omega,v)}
|\chi(\omega,v)| \leq K(\omega) + \varepsilon |\x(v)|^{\frac12+\delta}
\quad	\text{ for all } v \in \Cluster_\omega\ \Prmp\text{-almost surely.}
\end{equation}
Further, there is a random variable $D \in L^2(\Prob)$ such that
\begin{equation}	\label{eq:pathwise bound on chi(Y_k)}
|\chi(Y_k)| \leq D + k^{\frac14+\delta}
\quad \Prob\text{-almost surely for all } k \in \N.
\end{equation}
\end{lemma}
\begin{proof}
For $k \in \Z$, set $\eta_k \defeq L_k - \frac{\Ermp[L_1]}{\Ermp[C_1^{-1}]} \frac{1}{C_k}$.
Then, for $n \in \N$,
\begin{align*}
\chi(R^{\mathrm{pre}}_n)
&=	\textstyle
\x(R^{\mathrm{pre}}_n) - \psi(R^{\mathrm{pre}}_n)
= \sum_{k=1}^n \big(L_k - \frac{\Ermp[L_1]}{\Ermp[C_1^{-1}]} \frac{1}{C_k}\big) + \frac{\Ermp[L_1]}{\Ermp[C_1^{-1}]} \phi(\omega,\mathbf{0})	\\
&\eqdef	\textstyle
\sum_{k=1}^n \eta_k + \frac{\Ermp[L_1]}{\Ermp[C_1^{-1}]} \phi(\omega,\mathbf{0})
\end{align*}
where $\eta_1,\ldots,\eta_n$ are i.i.d.\ centered random variables.
Here $\Ermp[e^{\vartheta |\eta_1|}] < \infty$ for some $\vartheta > 0$ by Lemma \ref{Lem:effective resistances}.
From \eqref{eq:bound on S_n}, the fact that $|\phi(\omega,\mathbf{0})| \leq 1/{C_0}$ and again
Lemma \ref{Lem:effective resistances}, which guarantees that $\Ermp[1/{C_0^2}]<\infty$,
we thus infer, for arbitrary given $\varepsilon \in (0,\frac14)$ and $\delta \in (0,\frac12)$,
\begin{equation}	\label{eq:chi(R_n^pre) bound}	\textstyle
|\chi(R^{\mathrm{pre}}_n)|
\leq  |\eta_1+\ldots+\eta_n| + \frac{\Ermp[L_1]}{\Ermp[C_1^{-1}]}\frac1{C_0}
\leq K_2 + \varepsilon n^{\frac12+\delta}
\end{equation}
for all $n \in \N$ and a random variable $K_2$ on $(\Omega,\F)$ satisfying $\Ermp[K_2^2]<\infty$.
If $v \in \omega_n$ for some $n > 0$,
then
\begin{equation*}	\textstyle
|\chi(R^{\mathrm{pre}}_n) - \chi(v)| \leq L_n + \frac{\Ermp[L_1]}{\Ermp[C_1^{-1}]} \frac{1}{C_n}
\eqdef \xi_n.
\end{equation*}
The $\xi_n$, $n \in \N$ are nonnegative and i.i.d.\ under $\Prmp$.
Hence, \eqref{eq:bound on xi_n} gives
\begin{equation}	\label{eq:bound on xi_n corrector}
\xi_n \leq K_1 + \varepsilon n^{\frac12+\delta}
\end{equation}
for all $n \in \N$, where $K_1$ is a nonnegative random variable on $(\Omega,\F)$
with $\Erm[K_1^2] < \infty$.
Analogous arguments apply when $v \in \omega_n$ with $n \leq 0$.
Combining \eqref{eq:chi(R_n^pre) bound} and \eqref{eq:bound on xi_n corrector},
we infer that for sufficiently small $\delta > 0$
there is a random variable $K \geq 1$ on $(\Omega,\F)$ with $\Ermp[K^2]<\infty$
such that
\begin{equation*}
|\chi(\omega,v)| \leq K(\omega) + 2 \varepsilon |\x(v)|^{\frac12+\delta}
\end{equation*}
for all $v \in \Cluster_\omega$, i.e., \eqref{eq:pathwise bound on chi(omega,v)} holds.

Finally, we turn to the proof of \eqref{eq:pathwise bound on chi(Y_k)}.
We use \eqref{eq:pathwise bound on chi(omega,v)} with $\varepsilon = \frac12$ twice
and $|\x(Y_k)| \leq k$ to infer
for some sufficiently small $\delta \in (0,\frac12)$,
\begin{align}	\textstyle
|\chi(Y_k)|
&\leq K(\omega) + \tfrac12 |X_k|^{\frac12+\delta} = K(\omega) + \tfrac12 |\psi(Y_k) + \chi(Y_k)|^{\frac12+\delta}	\notag	\\
&\leq K(\omega) + \tfrac12 |\psi(Y_k)|^{\frac12 + \delta} + \tfrac12 |\chi(Y_k)|^{\frac12+\delta}	\notag	\\
&\leq K(\omega) + \tfrac12 |\psi(Y_k)|^{\frac12 + \delta} + \tfrac12 K(\omega)^{\frac12+\delta} + \tfrac12 k^{(\frac12+\delta)^2} 	\label{eq:bound chi(Y_k) almost final}
\end{align}
$P_{\omega}$-a.\,s.\ for all $k \in \N_0$.
Now $(\psi(\omega,Y_k))_{k \in \N_0}$ is a martingale with respect to $P_\omega$
and also a $\Prob$-martingale with respect to its canonical filtration.
As it has bounded increments, we may apply Lemma \ref{Lem:moderate deviations estimates}(d)
to infer the existence of a variable $K_2 \in L^2(\Prob)$ such that
$|\psi(\omega,Y_k)| \leq K_2 + k^{\frac12 + \delta}$ for all $k \in \N_0$ $\Prob$-a.\,s.
Using this in \eqref{eq:bound chi(Y_k) almost final}, we conclude
\begin{align*}
|\chi(Y_k)|
&\leq K(\omega) + \tfrac12 K_2^{\frac12 + \delta} + \tfrac12 k^{(\frac12 + \delta)^2} + \tfrac12 K(\omega)^{\frac12+\delta} + \tfrac12 k^{(\frac12+\delta)^2}	\\
&\leq K(\omega) + \tfrac12 K(\omega)^{\frac12+\delta} + K_2^{\frac12 + \delta} + k^{(\frac12 + \delta)^2}
\end{align*}
$\Prob$-a.\,s.\ for all $k \in \N_0$.
The assertion now follows with
$D(\omega,(v_k)_{k \in \N_0}) \defeq K(\omega) + \tfrac12 K(\omega)^{\frac12+\delta} + K_2(\omega,(v_k)_{k \in \N_0})^{\frac12 + \delta}$
and with the observation that $(\frac12 + \delta)^2 \to \frac14$ as $\delta \downarrow 0$.
\end{proof}

\section{Quenched joint functional limit theorem via the corrector method}	\label{sec:corrector}

From Proposition \ref{Prop:corrector} and the martingale central limit theorem,
we infer the following result.

\begin{proposition}	\label{Prop:martingale CLT}
For $\Prmp$-almost all $\omega \in \Omega$,
\begin{equation}	\label{eq:joint invariance principle psi,M}	\textstyle
\frac1{\sqrt{n}} (\psi(Y_{\lfloor nt\rfloor}),M_{\lfloor nt \rfloor})		\Rightarrow	(B,M)
\quad	\text{under } P_{\omega}
\end{equation}
where $(B,M)$ is a two-dimensional centered Brownian motion with covariance matrix
$\Sigma = (\sigma_{ij})_{i,j=1,2}$ of the form
\begin{equation*}
\begin{pmatrix}
\sigma_{11}	&	\sigma_{12}	\\
\sigma_{21}	&	\sigma_{22}
\end{pmatrix}
=
\begin{pmatrix}
\E[\psi(Y_1)^2]							&	\E[\psi(Y_1) \nu_{\omega}(\mathbf{0},Y_1)]	\\
\E[\psi(Y_1) \nu_{\omega}(\mathbf{0},Y_1)]	&	\E[\nu_{\omega}(\mathbf{0},Y_1)^2]
\end{pmatrix}.
\end{equation*}
\end{proposition}
\begin{proof}
Let $\alpha,\beta \in \R$.
We prove that $n^{-\frac12}(\alpha \psi(Y_{\lfloor nt\rfloor}) + \beta M_{\lfloor nt \rfloor})$
converges to a centered Brownian motion in the Skorokhod space $D([0,1])$.
To this end, we invoke the martingale functional central limit theorem \cite[Theorem 3.2]{Helland:1982}.
Let
\begin{equation*}	\textstyle
\xi_{n,k}(\alpha,\beta) \defeq \frac{\alpha}{\sqrt{n}} (\psi(Y_k)-\psi(Y_{k-1})) + \frac{\beta}{\sqrt{n}} \nu_{\omega}(Y_{k-1},Y_k),
\end{equation*}
for $n \in \N$ and $k=1,\ldots,n$.
In order to apply the result, it suffices to check the following two conditions:
\begin{align}
\sum_{k=1}^{\lfloor nt \rfloor} E_{\omega}\big[\xi_{n,k}(\alpha,\beta)^2 \mid \G_{k-1}\big]
\to c(\alpha,\beta) t	\quad	\text{in } P_{\omega}\text{-probability}	\label{eq:Helland (3.4)}\\
\sum_{k=1}^{\lfloor nt \rfloor} E_{\omega}\big[\xi_{n,k}(\alpha,\beta)^2 \1_{\{|\xi_{n,k}(\alpha,\beta)| > \varepsilon\}} \mid \G_{k-1}\big]
\to 0	\quad	\text{in } P_{\omega}\text{-probability}	\label{eq:Helland (3.5)}
\end{align}
for every $\varepsilon > 0$
where $c(\alpha,\beta)$ is a suitable constant depending on $\alpha,\beta$.
Here, $\G_k \defeq \sigma(Y_0,\ldots,Y_k) \subset \G$ where $Y_0,Y_1,\ldots$ are considered
as functions $V^{\N_0} \to V$.
To check \eqref{eq:Helland (3.4)}, we first define the functions $f,g,h: \Omega \to \R$,
\begin{align*}
f(\omega)	&\defeq E_{\omega}[\psi(\omega,Y_1)^2],	\\
g(\omega)	&\defeq E_{\omega}[\psi(\omega,Y_1) \nu_{\omega}(\mathbf{0},Y_1)],	\\
h(\omega) &\defeq E_{\omega}[\nu_{\omega}(\mathbf{0},Y_1)^2].
\end{align*}
These three functions are finite
and integrable with respect to $\Prmp$.
This follows from Proposition \ref{Prop:corrector}(c) for $\psi$ and from the boundedness of $\nu_{\omega}(v,w)$
as a function of $\omega,v,w$.
From Proposition \ref{Prop:corrector}(b) (with $u=Y_{k-1}$ and $v= Y_k-Y_{k-1}$), we infer for every $k \in \N$:
\begin{align*}
f(\overline{\omega}(k-1))
&=
E_{\omega}\big[(\psi(\omega,Y_k)-\psi(\omega,Y_{k-1}))^2 \mid \G_{k-1}\big],	\\
g(\overline{\omega}(k-1))
&= 
E_{\omega}\big[(\psi(Y_k)-\psi(Y_{k-1})) \nu_{\omega}(Y_{k-1},Y_k) \mid \G_{k-1}\big],	\\
h(\overline{\omega}(k-1))
&= 
E_{\omega}\big[\nu_{\omega}(Y_{k-1},Y_k)^2 \mid \G_{k-1}\big].
\end{align*}
Lemma \ref{Lem:Birkhoff quenched} thus gives
\begin{align}
&		\textstyle
\sum_{k=1}^{\lfloor nt \rfloor} E_{\omega}\big[\xi_{n,k}(\alpha,\beta)^2 \mid \G_{k-1}\big]		\notag	\\
&~=		\textstyle
\frac1n \sum_{k=1}^{\lfloor nt \rfloor} \big(\alpha^2 f(\overline{\omega}(k-1)) + 2\alpha\beta g(\overline{\omega}(k-1)) + \beta^2 h(\overline{\omega}(k-1))\big)	\notag	\\
&~\to	\textstyle
\big(\alpha^2 \Ermp[f] + 2\alpha\beta \Ermp[g] + \beta^2 \Ermp[h]\big)t.	\label{eq:covariances}
\end{align}
This gives \eqref{eq:Helland (3.4)} with $P_\omega$-almost sure convergence instead
of the weaker convergence in $P_\omega$-probability.
Equation \eqref{eq:Helland (3.5)} follows by an argument in the same spirit.
We therefore conclude that
\begin{equation*}	\textstyle
\frac{\alpha}{\sqrt{n}} \psi(Y_{\lfloor nt\rfloor}) + \frac{\beta}{\sqrt{n}} M_{\lfloor nt \rfloor} \Rightarrow
\alpha B(t) + \beta M(t)
\end{equation*}
in the Skorokhod space $D[0,1]$ as $n \to \infty$.
This implies convergence of the finite-dimensional distributions and, hence,
by the Cram\'er-Wold device,
we conclude that the finite-dimensional distributions of
$n^{-1/2}(\psi(Y_{\lfloor nt\rfloor}),M_{\lfloor nt \rfloor})$
converge to the finite-dimensional distributions of $(B,M)$.
As the sequences $n^{-1/2} \psi(Y_{\lfloor nt\rfloor})$ und $n^{-1/2} M_{\lfloor nt \rfloor}$
are tight in the Skorokhod space $D[0,1]$, so is $n^{-1/2} (\psi(Y_{\lfloor nt\rfloor}),M_{\lfloor nt \rfloor})$,
cf.\;\cite[Section 15]{Billingsley:1968}.
This implies \eqref{eq:joint invariance principle psi,M}.
The formula for the covariances follows from \eqref{eq:covariances}.
\end{proof}

We can now give the proof of Theorem \ref{Thm:joint CLT}.

\begin{proof}[of Theorem \ref{Thm:joint CLT}]
In view of \eqref{eq:joint invariance principle psi,M} and Theorem 4.1 in \cite{Billingsley:1968},
it suffices to check that, for $\Prmp$-almost all $\omega$,
\begin{equation}
\max_{k=0,\ldots,n} \frac{|\chi(\omega,Y_k)|}{\sqrt{n}}	\to	0	\quad	\text{in } P_\omega	\text{-probability.}
\end{equation}
For a sufficiently small $\delta \in (0,\frac12)$,
we conclude from \eqref{eq:pathwise bound on chi(omega,v)},
with $K$ slightly increased if necessary (it is sufficient to replace the original $K$ by $2K + 1$),
that
\begin{equation*}
|\chi(\omega,Y_k)| \leq K(\omega) + 2\varepsilon |\psi(\omega,Y_k)|^{\frac12+\delta} \text{ a.\,s.}
\end{equation*}
for all $k \in \N_0$.
By Proposition \ref{Prop:martingale CLT},
$\max_{k=0,\ldots,n} |\psi(Y_k)|/\sqrt{n}$ converges in distribution
to the maximum $\max_{0 \leq t \leq 1} |B(t)|$ of the absolute value of a Brownian motion on the unit interval.
Hence
\begin{equation}	\label{eq:chi bdd by psi}
\frac{\max_{k=0,\ldots,n} |\chi(\omega,Y_k)|}{\sqrt{n}}
\leq \frac{K(\omega)}{\sqrt{n}} + 2\varepsilon \bigg(\frac{\max_{k=0,\ldots,n} |\psi(\omega,Y_k)|}{\sqrt{n}}\bigg)^{\!\!\frac12+\delta} \frac{1}{n^{1/4-\delta/2}}
\end{equation}
which converges to $0$ in distribution by Slutsky's theorem, and hence in probability.

It remains to prove \eqref{eq:E[BM]=E[B^2]},
that is, $\E[B(1)M(1)] = \E[B(1)^2] = \sigma^2$.
Uniform integrability, see \eqref{eq:sup exp M} and \eqref{eq:max X_k^2 bound} below, implies
\begin{equation}	\label{eq:covariance matrix as moment limits}
\E[B(1)M(1)] = \lim_{n \to \infty} \frac1n \E[X_n M_n]
\quad	\text{and}	\quad
\E[B(1)^2] = \lim_{n \to \infty} \frac1n \E[X_n^2].
\end{equation}
It follows from \eqref{eq:pathwise bound on chi(omega,v)} and \eqref{eq:covariance matrix as moment limits}
that $\frac1n \E[\chi(Y_n)^2] \to 0$.
Thus, from $X_n = \psi(Y_n)+\chi(Y_n)$, we deduce
\begin{align*}
\lim_{n \to \infty} \frac1n \E[X_n M_n] = \! \lim_{n \to \infty} \frac1n \E[\psi(Y_n) M_n]
\text{ and }
\lim_{n \to \infty} \frac1n \E[X_n^2] = \! \lim_{n \to \infty} \frac1n \E[\psi(Y_n)^2].
\end{align*}
It thus remains to show
\begin{equation}	\label{eq:covariance statement in terms of psi}
\lim_{n \to \infty} \frac1n \E[\psi(Y_n) M_n]	= \lim_{n \to \infty} \frac1n \E[\psi(Y_n)^2].
\end{equation}
Since $(\psi(Y_n))_{n \in \N_0}$ and $(M_n)_{n \in \N_0}$ are martingales with respect to the same filtration,
their increments at different times are uncorrelated. Consequently,
\begin{equation}	\label{eq:E[psi(Y_n)M_n]}
\frac1n \E[\psi(Y_n) M_n] = \frac1n \sum_{k=1}^n \E\big[(\psi(Y_k)-\psi(Y_{k-1}))\nu_{\omega}(Y_{k-1},Y_k)\big].
\end{equation}
For $\omega \in \Omega$ and $w \in N_{\omega}(v) = \{w \in V: p_{\omega,0}(v,w) > 0\}$,
an elementary calculation yields
\begin{equation}	\label{eq:nu_omega(v,w) ausgewertet}
\nu_{\omega}(v,w)
= \begin{cases}
\x(w)-\x(v)											&	\text{if }	w \not = v,	\\
\sum_{u \sim v: \omega(\langle u,v \rangle)=1} (\x(u)-\x(v))		&	\text{if }	w = v.
\end{cases}
\end{equation}
In particular, $\nu_{\omega}(Y_{k-1},Y_k) = X_k-X_{k-1}$ if $Y_k \not = Y_{k-1}$.
Therefore, we can replace $\nu_{\omega}(Y_{k-1},Y_k)$ by $X_k-X_{k-1}$ in \eqref{eq:E[psi(Y_n)M_n]}
and infer
\begin{align*}
\frac1n \E&[\psi(Y_n) M_n]
= \frac1n \sum_{k=1}^n \E\big[(\psi(Y_k)-\psi(Y_{k-1}))(X_k-X_{k-1})\big]	\\
&= \frac1n \sum_{k=1}^n \E\big[(\psi(Y_k)-\psi(Y_{k-1}))(\psi(Y_k)-\psi(Y_{k-1})+\chi(Y_k)-\chi(Y_{k-1}))\big]	\\
&= \frac1n \E[\psi(Y_n)^2] + \frac1n \sum_{k=1}^n \E\big[(\psi(Y_k)-\psi(Y_{k-1}))(\chi(Y_k)-\chi(Y_{k-1}))\big].
\end{align*}
The second sum vanishes as $n \to \infty$ since $(\psi(Y_n))_{n \in \N_0}$ has bounded increments
and $\frac1n\E[|\chi(Y_n)|] \to 0$ as $n \to \infty$.
\end{proof}

\section{The proof of the Einstein relation}\label{sec:proof einstein}

\subsection{Proof of Equation \eqref{eq:supsecondmoment}.}

For the proof of the Einstein relation, we use a combination of the approaches from \cite{Gantert+al:2012,Guo:16,Lebowitz+Rost:94}.

\begin{lemma}	\label{Lem:max X_k^2 bound}
It holds that
\begin{equation}	\label{eq:max X_k^2 bound}
\limsup_{n \to \infty} \frac1n \E\Big[\max_{k=1,\ldots,n} X_k^2\Big] < \infty.
\end{equation}
\end{lemma}
\begin{proof}
By Lemma \ref{Lem:environment seen from the particle},
$(\overline{\omega}(n))_{n \geq 0}$ is a stationary and ergodic sequence under $\Prob$.
This sequence can be extended canonically to a two-sided stationary and ergodic sequence $(\overline{\omega}(n))_{n \in \Z}$
on the underlying space $\Omega_{\mathbf{0}}^\Z$.
The increment sequences $(X_n-X_{n-1})$ and $(Y_n - Y_{n-1})_{n \in \Z}$
also form stationary and ergodic sequences by Lemma \ref{Lem:ergodic theory input}.
Therefore, we can invoke Theorem 1 in \cite{Peligrad+al:2007} and conclude that
\begin{equation}	\label{eq:Peligrad}
\E\Big[\max_{k=1,\ldots,n} X_k^2\Big] \leq 4 n (1+80 \delta_{n,2})^2
\end{equation}
where $\delta_{n,2} = \sum_{j=1}^{n} j^{-\frac32} \big(\Ermp[E_\omega[X_j]^2]\big)^{\frac12}$.
Here, we have
\begin{equation*}	\textstyle
E_\omega[X_j]^2
= E_\omega[\psi(Y_j) + \chi(Y_j)]^2
= E_\omega[\chi(Y_j)]^2
\leq E_\omega[|\chi(Y_j)|^2].
\end{equation*}
There exists a random variable $D \in L^2(\Prob)$
such that \eqref{eq:pathwise bound on chi(Y_k)} holds for some $\delta > 0$ sufficiently small,
i.e., $|\chi(Y_j)| \leq D + j^{\frac14+\frac\delta2}$ for all $j \in \N_0$ $\Prob$-a.\,s.
Hence,
\begin{align*}
E_\omega[|\chi(Y_j)|^2]
&\leq E_\omega[|D + j^{\frac14+\frac\delta2}|^2]
\leq E_\omega[2 D^2] + 2j^{\frac12+\delta}.
\end{align*}
Consequently,
\begin{align*}
\sup_{n \geq 1} \delta_{n,2}
\leq \sum_{j\geq1} j^{-\frac32} \big(2 \E[D^2] + 2j^{\frac12+\delta}]\big)^{\frac12}
< \infty 
\end{align*}
if we pick $\delta > 0$ sufficiently small.
Consequently, \eqref{eq:max X_k^2 bound} follows from \eqref{eq:Peligrad}.
\end{proof}

\subsection{Proof of Equation \eqref{eq:speed approx by covariance}.}

The first two steps of the proof of Theorem \ref{Thm:Einstein relation} are completed.
We continue with Step 3, i.e., the proof of \eqref{eq:speed approx by covariance}.
It is based on a second order Taylor expansion for
$\sum_{j=1}^n \log p_{\omega,\lambda}(Y_{j-1},Y_{j})$ at $\lambda = 0$:
\begin{align}
&
\sum_{j=1}^n \log p_{\omega,\lambda}(Y_{j-1},Y_{j})	\notag	\\
&~=		\lambda M_n	 + \frac{\lambda^2}{2} \sum_{j=1}^n \! \Big(
\frac{p_{\omega,0}''(Y_{j-1},Y_{j})}{p_{\omega,0}(Y_{j-1},Y_{j})} - \nu_{\omega}(Y_{j-1},Y_j)^2 \!\Big)
+ \lambda^2 \sum_{j=1}^n o_{\omega,\lambda}(Y_{j-1},Y_j)	\label{eq:Taylor}
\end{align}
where $o_{\omega,\lambda}(v,w)$ tends to $0$ uniformly in $\omega \in \Omega$ and $v,w \in V$ as $\lambda \to 0$.
Set
\begin{equation*}
A_n
\defeq \frac12 \sum_{j=1}^{n} \bigg(\nu_{\omega}(Y_{j-1},Y_{j})^{2} - \frac{p_{\omega}''(Y_{j-1},Y_j)}{p_{\omega}(Y_{j-1},Y_j)}\bigg)
\end{equation*}
where we write $p_\omega$ for $p_{\omega,0}$,
and 
\begin{equation*}
R_{\lambda,n} \defeq	\lambda^2 \sum_{j=1}^{n} o_{\omega,\lambda}(Y_{j-1},Y_{j}).
\end{equation*}
Both, $A_n$ and $R_{\lambda,n}$ are random variables on $(\Omega \times V^{\N_0},\F \otimes \G)$.

\begin{lemma}	\label{Lem:2nd order term}
Let $\lambda\to 0$ and $n\to\infty$
such that $\lim_{n \to \infty} \lambda^{2}n = \alpha \geq 0$.
Then
\begin{equation}	\label{eq:A asymptotics}
\lambda^2 A_n \to \frac{\alpha}{2} \E[M(1)^2]
\quad	\Prob\text{-a.\,s.\ and in } L^1
\end{equation}
and $R_{\lambda,n} \to 0$ $\Prob$-a.\,s.
\end{lemma}
\begin{proof}
The convergence $R_{\lambda,n} \to 0$ follows from
the fact that $o_{\omega,\lambda}(v,w)$ tends to $0$ uniformly in $\omega \in \Omega$ and $v,w \in V$ as $\lambda \to 0$.

For the proof of \eqref{eq:A asymptotics},
notice that $A_n-A_{n-1}$ is a function of $(\overline{\omega}(n-1),\overline{\omega}(n))$.
To make this more transparent, we write
\begin{align*}
A_n&-A_{n-1} = \frac12 \bigg(\nu_{\omega}(Y_{n-1},Y_{n})^{2} - \frac{p_{\omega}''(Y_{n-1},Y_n)}{p_{\omega}(Y_{n-1},Y_n)}\bigg)	\\
&= 
\frac12 \bigg(\nu_{\overline{\omega}(n-1)}(\mathbf{0},\varphi(\overline{\omega}(n-1),\overline{\omega}(n)))^{2}
- \frac{p_{\overline{\omega}(n-1)}''(\mathbf{0},\varphi(\overline{\omega}(n-1),\overline{\omega}(n)))}
{p_{\overline{\omega}(n-1)}(\mathbf{0},\varphi(\overline{\omega}(n-1),\overline{\omega}(n)))}\bigg)
\end{align*}
with the function $\varphi$ from the proof of Lemma \ref{Lem:ergodic theory input}.
Since $(\overline{\omega}(n))_{n \in \N_0}$ is ergodic under $\Prob$,
so is $(A_n-A_{n-1})_{n \geq 1}$, see e.g.\ Lemma 5.6(c) in \cite{Axelson-Fisk+H"aggstr"om:2009b}.
Birkhoff's ergodic theorem gives
\begin{equation*}
\lim_{n \to \infty} \frac1n A_n
= \frac12 \E\bigg[\sum_{w\in N_{\omega}(\mathbf{0})} \bigg(\nu_{\omega}(\mathbf{0},w)^{2} - \frac{p_{\omega}''(\mathbf{0},w)}{p_{\omega}(\mathbf{0},w)}\bigg) p_{\omega}(\mathbf{0},w)\bigg]
\quad	\Prob \text{-a.\,s.}
\end{equation*} 
Further, for all $v$ and all $\omega$, $p_{\omega}(v,\cdot)$
is a probability measure on $N_{\omega}(v)$, hence
$\sum_{w\in N_{\omega}(v)} p_{\omega}''(v,w) = 0$.
Consequently,
\begin{equation*}
\lim_{n \to \infty} \frac1n A_n
=
\frac12 \E\bigg[\sum_{w\in N_{\omega}(\mathbf{0})} \nu_{\omega}(\mathbf{0},w)^{2} p_{\omega}(\mathbf{0},w)\bigg]
\quad	\Prob \text{-a.\,s.}
\end{equation*}
On the other hand, by Theorem \ref{Thm:joint CLT}, we have
\begin{align*}
\E[M(1)^2]
&=
\lim_{n \to \infty} \frac1n \E[M_n^2]
= \lim_{n \to \infty} \frac1n \sum_{k=1}^n \E[\nu_{\omega}(Y_{k-1},Y_{k})^2]	\\
&=
\E[\nu_{\omega}(Y_{0},Y_{1})^2]
= \E\bigg[\sum_{w\in N_{\omega}(\mathbf{0})} \nu_{\omega}(\mathbf{0},w)^{2} p_{\omega, 0}(\mathbf{0},w)\bigg],
\end{align*}
where the second equality follows from the fact that the increments of square-integrable martingales are uncorrelated
and the third equality follows from the fact that $(\nu_{\omega}(Y_{k-1},Y_{k}))_{k \in \N}$ is an ergodic sequence
under $\Prob$.
\end{proof}

\begin{proposition}	\label{Prop:3rd step}
For any $\alpha > 0$, it holds that
\begin{equation*}	\tag{\ref{eq:speed approx by covariance}}
\lim_{\substack{\lambda \to 0,\\ \lambda^2 n \to \alpha}} \frac{\E_{\lambda}[X_n]}{\lambda n}
= \E[B(1) M(1)].
\end{equation*}
\end{proposition}
\begin{proof}
We follow Lebowitz and Rost \cite{Lebowitz+Rost:94}
and use the (discrete) Girsanov transform introduced in Section \ref{sec:main results}.
Indeed, using \eqref{eq:Taylor}, we get
\begin{align*}
\E_{\lambda}[X_n]
&=
\E\bigg[X_n
\exp\bigg(\sum_{j=1}^{n} \log \frac{p_{\omega,\lambda}(Y_{j-1},Y_{j})}{p_{\omega}(Y_{j-1},Y_{j})}\bigg)\bigg]	\\
&=
\E\big[X_n \exp\big(\lambda M_n - \lambda^2 A_n + R_{\lambda,n}\big)\big].
\end{align*}
Now divide by $\lambda n \sim \sqrt{\alpha n}$ and use Theorem \ref{Thm:joint CLT}, Lemma \ref{Lem:2nd order term},
Slutsky's theorem and the continuous mapping theorem to conclude that
\begin{align}	\label{eq:convergence in distribution for 3rd step}
\frac{X_n}{\lambda n}
e^{\lambda M_n - \lambda^2 A_n + R_{\lambda,n}}
\distto \frac{1}{\sqrt{\alpha}} B(1) e^{\sqrt{\alpha} M(1) - \frac\alpha2 \E[M(1)^2]}.
\end{align}
Suppose that along with convergence in distribution, convergence of the first moment holds.
Then we infer
\begin{align*}
\lim_{n \to \infty} \frac{\E_{\lambda}[X_n]}{\lambda n}
&= \frac{1}{\sqrt{\alpha}} \E\Big[B(1) \exp\Big(\sqrt{\alpha} M(1) - \frac\alpha2 \E[M(1)^2]\Big)\Big]
= \E[B(1) M(1)]
\end{align*}
where the last step follows from the integration by parts formula for two-dimensional Gaussian vectors.
It remains to show that the family on the left-hand side of \eqref{eq:convergence in distribution for 3rd step} is uniformly integrable.
To this end, use H\"older's inequality to obtain
\begin{align*}
\sup_{\lambda,n} & ~\E \bigg[\bigg|\frac{X_n}{\lambda n}
e^{\lambda M_n - \lambda^2 A_n + R_{\lambda,n}}\bigg|^{\frac65}\bigg]	\\
& \leq \sup_{\lambda,n} \E \bigg[\bigg|\frac{X_n}{\lambda n}\bigg|^{2}\bigg]^{\frac{3}{5}}
\sup_{\lambda,n} \E \Big[e^{3\lambda M_n - 3\lambda^2 A_n
+ 3 R_{\lambda,n}}\Big]^{\frac25}\!\!.
\end{align*}
By Lemma \ref{Lem:max X_k^2 bound}, the first supremum in the last line is finite.
Concerning the finiteness of the second,
notice that $\lambda^2 A_n$ and $R_{\lambda,n}$ are bounded
when $\lambda^2n$ stays bounded
(see the proof of Lemma \ref{Lem:2nd order term} for details),
whereas \eqref{eq:sup exp M} gives $\sup_{\lambda,n} \E [e^{3\lambda M_n}] < \infty$.
\end{proof}

\subsection{Regeneration points and times}	\label{subsec:regeneration times}

Given $\lambda \in (0,1]$,
define $\lambda$-dependent pre-regeneration points by:
\begin{equation*}
R^{\mathrm{pre},\lambda}_{n}= R^{\mathrm{pre}}_{n \lfloor 1/\lambda \rfloor},	\quad	n \in \Z.
\end{equation*}
The set of $\lambda$-pre-regeneration points is denoted by $\Rprelambda$.
The cluster is decomposed into independent pieces $\omega_n^\lambda \defeq [R^{\mathrm{pre},\lambda}_{n-1},R^{\mathrm{pre},\lambda}_{n})$, $n \in \Z$.
The $\lambda$-regeneration times are defined as
$\tau^{\lambda}_{0} \defeq \rho^{\lambda}_{0} \defeq 0$
and, inductively,
\begin{align*}
\tau_{n}^{\lambda}
&\defeq
\inf\{k > \tau^{\lambda}_{n-1}: Y_{k}\in \Rprelambda,
Y_{j} \neq Y_{k} \text{\;for\;all\;} j < k \;\text{and}	\\
&\hphantom{\defeq \inf\{k > \tau^{\lambda}_{n-1}: \,}
Y_j \not \in \Rprelambda \text{ for all } j \geq k \text{ with } X_j < X_k \},	\\
R^{\lambda}_{n}
&\defeq
Y_{\tau^{\lambda}_{n}}
\end{align*}
for $n \in \N$.
We further set $\rho_n^\lambda \defeq X_{\tau^{\lambda}_{n}} = \x(R^{\lambda}_{n})$.
In words, a $\lambda$-regeneration point is a $\lambda$-pre-regeneration point $R^{\mathrm{pre},\lambda}_{n}$
such that the walk after the first visit to $R^{\mathrm{pre},\lambda}_{n}$ never returns to $R^{\mathrm{pre},\lambda}_{n-1}$,
the $\lambda$-pre-regeneration point to the left.
In the context of regeneration-time arguments
it will be useful at some points to work with a different percolation law than $\Prmp$ or $\Prmp^*$,
namely, the cycle-stationary percolation law $\Prmp^\circ$,
which is defined below.

\begin{definition}	\label{def:cycle-stationary percolation law}
The \emph{cycle-stationary percolation law} $\Prmp^\circ$
is defined to be the unique probability measure on $(\Omega,\F)$ such that the
cycles $\omega_n$, $n \in \Z$ are i.i.d.\ under $\Prmp^\circ$
and such that each $\omega_n$ has the same law under $\Prmp^\circ$
as $\omega_1$ under $\Prmp^*$.
We write $\Prob^\circ_\lambda$ for $\Prob_{\Prmp^\circ,\lambda}^{\mathbf{0}}$.
\end{definition}

We define
\begin{equation*}
\H_n^{\lambda}	~\defeq~
\sigma(\tau_1^{\lambda},\ldots,\tau_n^{\lambda}, Y_0, Y_1, \ldots, Y_{\tau_n^{\lambda}},
\omega_k^\lambda: \x(R^{\mathrm{pre},\lambda}_{k}) \leq \rho_n^{\lambda}),
\end{equation*}
the $\sigma$-algebra of the walk up to time $\tau_n^{\lambda}$ and of the environment up to $\rho_n^{\lambda}$.
The distances between $\lambda$-regeneration times are not i.i.d.,
but $1$-dependent.

\begin{lemma}	\label{Lem:lambda-regeneration times and points}
For any $n \in \N$ and all measurable sets $F \in \F_{\geq} = \sigma(p_{\langle v,w \rangle}: \x(v) \wedge \x(w) \geq 0)$ and $G \in \G$,
\begin{align}
\Prob_{\lambda}((\theta^{R^{\lambda}_n} & Y_{\tau_n^{\lambda}+k})_{k \geq 0} \in G,\,\theta^{R_{n-1}^{\lambda}} \omega \in F \mid \H_{n-1}^{\lambda})	\notag	\\
&\quad= \Prob_{\lambda}^\circ((Y_{k})_{k \geq 0} \in G,\, \omega \in F \mid X_k > \x(R^{\mathrm{pre},\lambda}_{-1}) \text{ for all } k \in \N).
\label{eq:lambda regeneration times and points}
\end{align}
In particular, $((\tau_{n+1}^{\lambda}-\tau_n^{\lambda}, \rho_{n+1}^{\lambda}-\rho_{n}^{\lambda}))_{n \in \N}$
is a $1$-dependent sequence of random variables under $\Prob_{\lambda}$.
\end{lemma}

Since Lemma \ref{Lem:lambda-regeneration times and points} is a natural observation and its proof is a rather straightforward
but tedious adaption of the proof of Lemma 4.1 in \cite{Gantert+al:2018}, we omit the details of the proof.

The subsequent lemma provides the key estimate for the distances between $\lambda$-regeneration points.

\begin{lemma}	\label{Lem:uniform exponential regeneration point estimates}
There exist finite constants $C,\varepsilon>0$ depending only on $p$ such that, for every sufficiently small $\lambda > 0$,
\begin{equation}	\label{eq:uniform exponential regeneration point estimates}
\Prob_{\lambda}(\rho_{2}^\lambda-\rho_1^\lambda \geq n)	~\leq~ C e^{- \lambda \varepsilon n}	\quad	\text{for all } n \in \N_0.
\end{equation}
In particular,
\begin{equation}	\label{eq:lambda^2 E_lambda X_tau_1^2<infty}
\limsup_{\lambda \to 0} \lambda^2 \E_{\lambda}[(\rho_1^\lambda)^2] < \infty,
\qquad
\limsup_{\lambda \to 0} \lambda^2 \E_{\lambda}[(\rho_2^\lambda-\rho_1^\lambda)^2] < \infty,
\end{equation}
and
\begin{equation}	\label{eq:asymptotics after Cauchy-Schwarz trick}
\limsup_{\lambda \to 0} \lambda^2 \sum_{n \geq 1} n \Prob_\lambda(\rho_2^\lambda-\rho_1^\lambda \geq n)^{\frac12} < \infty.
\end{equation}
\end{lemma}

For the proof, we require the following lemma.

\begin{lemma}	\label{Lem:LDP upper bound for lambda-pre-regeneration points}
There exist finite constants $\varepsilon=\varepsilon(p)>0$, $c_*=c_*(p) > 0$ such that, for all $x \geq 0$,
\begin{align*}
\Prmp^\circ(\x(R^{\mathrm{pre}}_{\lfloor \varepsilon x\rfloor }) > x)
\leq e^{-c_* x}.
\end{align*}
\end{lemma}
\begin{proof}
It follows from Lemma 3.3(b) of \cite{Gantert+al:2018} that there exists a constant $c(p) \in (0,1)$
depending only on $p$ such that
\begin{equation*}
\Prmp^\circ(\x(R_1^{\mathrm{pre}}) > m) \leq c(p)^m
\end{equation*}
for all $m \in \N_0$. Hence, the moment generating function $\vartheta \mapsto \Ermp^\circ[e^{\vartheta \x(R_1^{\mathrm{pre}})}]$
is finite in some open interval containing the origin,
in particular, $\x(R_1^{\mathrm{pre}})$ has positive finite mean $\mu(p)$ (depending only on $p$).
Let $\varepsilon \in (0,\mu(p)^{-1})$.
Then, for some sufficiently small $u>0$,
\begin{equation*}
1 > \Ermp^\circ[e^{u \x(R_1^{\mathrm{pre}})}] e^{-u \varepsilon^{-1}} \eqdef e^{-c_* \varepsilon^{-1}}.
\end{equation*}
Fix $x \geq 0$.
Since the $\omega_n$, $n \in \N$ are i.i.d.\ under $\Prmp^\circ$,
$\x(R^{\mathrm{pre}}_{\lfloor \varepsilon x \rfloor})$ is the sum of $\lfloor \varepsilon x \rfloor$
i.i.d.\ random variables each having the same law as $\x(R_1^{\mathrm{pre}})$ under $\Prmp^\circ$.
Consequently, Markov's inequality gives
\begin{align*}
\Prmp^\circ(\x(R^{\mathrm{pre}}_{\lfloor \varepsilon x\rfloor }) > x)
&\leq \Ermp^\circ[e^{u \x(R_1^{\mathrm{pre}})}]^{\lfloor \varepsilon x \rfloor} e^{-ux}
\leq \Ermp^\circ[e^{u \x(R_1^{\mathrm{pre}})}]^{\varepsilon x} e^{-ux}
\leq e^{-c_* x}.
\end{align*}
\end{proof}

\begin{proof}[of Lemma \ref{Lem:uniform exponential regeneration point estimates}]
We first derive \eqref{eq:lambda^2 E_lambda X_tau_1^2<infty} and \eqref{eq:asymptotics after Cauchy-Schwarz trick}
from \eqref{eq:uniform exponential regeneration point estimates}.
We only prove the second relation of \eqref{eq:lambda^2 E_lambda X_tau_1^2<infty}.
For $\lambda>0$, summation by parts and \eqref{eq:uniform exponential regeneration point estimates} give
\begin{align*}
\lambda^2 \E_{\lambda}[(\rho_2^\lambda-\rho_1^\lambda)^2]
&= \lambda^2 \sum_{n \geq 0} (2n+1) \Prob_\lambda(\rho_2^\lambda-\rho_1^\lambda > n)
\leq C \lambda^2 \sum_{n \geq 0} (2n+1) e^{- \lambda \varepsilon n},
\end{align*}
which remains bounded as $\lambda \downarrow 0$. Analogously,
\begin{align*}
\lambda^2 & \sum_{n \geq 1} n \Prob_\lambda(\rho_2^\lambda-\rho_1^\lambda \geq n)^{\frac12}
\leq C^{\frac12} \lambda^2 \sum_{n \geq 0} n e^{-\lambda \varepsilon n/2},
\end{align*}
which again remains bounded as $\lambda \to 0$ and thus gives \eqref{eq:asymptotics after Cauchy-Schwarz trick}.

We now turn to the proof of \eqref{eq:uniform exponential regeneration point estimates}.
By Lemma \ref{Lem:lambda-regeneration times and points}
the law of $\rho_{2}^\lambda-\rho_1^\lambda$ under $\Prob_\lambda$
is the same as the law of $\rho_1^\lambda$ under $\Prob_\lambda^\circ$
given that $(Y_n)_{n \geq 0}$ never visits $R^{\mathrm{pre},\lambda}_{-1}$:
\begin{equation}
\Prob_\lambda(\rho_{2}^\lambda-\rho_1^\lambda \in \cdot)
= \Prob_\lambda^\circ(\rho_1^\lambda \in \cdot \, | \, X_k > R^{\mathrm{pre},\lambda}_{-1} \text{ for all } k \in \N).
\end{equation}
Let $C \defeq \{X_k > \x(R^{\mathrm{pre},\lambda}_{-1}) \text{ for all } k \in \N\}$.
In order for $C^\comp$ to occur, the walk $(Y_n)_{n \in \N_0}$ must travel at least $1/\lambda$ steps to the left on the backbone
as the distance of $\mathbf{0} = R^{\mathrm{pre},\lambda}_{0}$ and $R^{\mathrm{pre},\lambda}_{-1}$
is at least $1/\lambda$.
From Lemma 6.3 in \cite{Gantert+al:2018}, we thus conclude that
\begin{equation}	\label{eq:lower bound regeneration}
\Prob_\lambda^\circ(C^\comp) \leq \frac{2(e^{2\lambda}-1)}{e^\lambda-1} \frac{1}{e^2-1}.
\end{equation}
As $\lambda \to 0$, the bound on the right-hand side tends to $4/(e^2-1) =  0.626\ldots$.
Hence, for all sufficiently small $\lambda > 0$, we have $\Prob_\lambda^\circ(C^\comp) \leq 2/3$,
and therefore, $\Prob_\lambda^\circ(C) \geq 1/3$.
Fix such a $\lambda$. Then
\begin{equation}
\Prob_\lambda^\circ(\rho_1^\lambda \in \cdot | C)
= \Prob_\lambda^\circ(C)^{-1} \Prob_\lambda^\circ(\rho_1^\lambda \in \cdot, C)
\leq 3 \Prob_\lambda^\circ(\rho_1^\lambda \in \cdot, C)
\leq 3 \Prob_\lambda^\circ(\rho_1^\lambda \in \cdot)
\end{equation}
and it thus remains to bound $\Prob_\lambda^\circ(\rho_1^\lambda \geq n)$ for $n \in \N_0$.

The basic idea is that $\rho_1^\lambda \geq n$ if
either there are unusually few $\lambda$-pre-regeneration points in $[0,n]$
or the walk $(Y_k)_{k \in \N_0}$ has to make too many excursions of length at least $\lfloor \frac1\lambda\rfloor$ to the left.
To turn this idea into a rigorous proof, we first observe that for $\varepsilon = \varepsilon(p)>0$ from Lemma \ref{Lem:LDP upper bound for lambda-pre-regeneration points},
we have
\begin{equation}	\label{eq:basic decomposition tails rho}
\Prob_\lambda^\circ(\rho_1^\lambda \geq n)
\leq \Prob_\lambda^\circ(\x(R^{\mathrm{pre},\lambda}_{\lfloor \varepsilon \lambda n \rfloor}) > n)
+ \Prob_\lambda^\circ(\rho_1^\lambda \geq  \x(R^{\mathrm{pre},\lambda}_{\lfloor \varepsilon \lambda n \rfloor})).
\end{equation}
The first probability on the right-hand side of \eqref{eq:basic decomposition tails rho} is bounded by
\begin{align*}
\Prob_\lambda^\circ(\x(R^{\mathrm{pre},\lambda}_{\lfloor \varepsilon \lambda n \rfloor}) > n)
&= \Prmp^\circ(\x(R^{\mathrm{pre},\lambda}_{\lfloor \varepsilon \lambda n \rfloor}) > n)
= \Prmp^\circ(\x(R^{\mathrm{pre}}_{\lfloor \varepsilon \lambda n\rfloor \cdot \lfloor \frac1\lambda\rfloor}) > n)		\\
&\leq \Prmp^\circ(\x(R^{\mathrm{pre}}_{\lfloor \varepsilon n\rfloor }) > n) \leq e^{-c_* n}	
\end{align*}
where we have used the elementary inequality $\lfloor a \rfloor \cdot \lfloor b \rfloor \leq \lfloor ab \rfloor $ for all $a,b > 0$
and then Lemma \ref{Lem:LDP upper bound for lambda-pre-regeneration points}.

We now turn to the second probability on the right-hand side of \eqref{eq:basic decomposition tails rho}.
Observe that a $\lambda$-pre-regeneration point
$R^{\mathrm{pre},\lambda}_{i}$ is a $\lambda$-regeneration point
iff after the first visit to it, the random walk $(Y_k)_{k \geq 0}$ never returns to $R^{\mathrm{pre},\lambda}_{i-1}$.
We define $Z_0',Z_1',Z_2',\ldots$ to be the sequence of indices of the $\lambda$-pre-regeneration points visited by $(Y_k)_{k \geq 0}$
in chronological order, i.e., $Z_j'=i$ if the $j$th visit of $(Y_k)_{k \geq 0}$ to $\Rprelambda$
is at the point $R^{\mathrm{pre},\lambda}_{i}$.
We then define $Z_0,Z_1,Z_2,\ldots$ to be the corresponding agile sequence,
that is, each multiple consecutive occurrence of a number in the string
is reduced to a single occurrence.  
For instance, if
\begin{equation*}
(Z_0',Z_1',Z_2',Z_3',Z_4',Z_5',Z_6',\ldots) = (0,-1,0,0,0,1,2,\ldots),
\end{equation*}
then
\begin{equation*}
(Z_0,Z_1,Z_2,Z_3,Z_4,\ldots) = (0,-1,0,1,2,\ldots).
\end{equation*}
Then $R^{\mathrm{pre},\lambda}_{i}$ is a $\lambda$-regeneration point if for the first $j \in \N$ with $Z_j = i$,
we have $Z_{k} \geq i$ for all $k \geq j$.
Let
\begin{equation*}
\varrho^* \defeq \inf\{Z_k: k \in \N \text{ with } Z_j < Z_k \leq Z_l \text{ for all } 0 \leq j < k \leq l\}.
\end{equation*}
Then
\begin{equation*}
\Prob_\lambda^\circ(\rho_1^\lambda \geq  \x(R^{\mathrm{pre},\lambda}_{\lfloor \varepsilon \lambda n \rfloor}))
= \Prob_\lambda^\circ(\varrho^* \geq \lfloor\varepsilon \lambda n\rfloor).
\end{equation*}
We compare the latter probability with the corresponding probability for a biased nearest-neighbor random walk on $\Z$
which at any vertex is more likely to move left than the walk $(Z_k)_{k \in \N_0}$.
More precisely, we may assume without loss of generality
that on the underlying probability space there exists a  biased nearest-neighbor random walk on $\Z$
which we denote by $(S_k)_{k \in \N_0}$
such that $\Prob_\lambda^\circ(S_0 = 0)=1$ and
\begin{equation*}
\tfrac6{10} \defeq \Prob_\lambda^\circ(S_{k+1} = j+1 \mid S_k = j) = 1 - \Prob_\lambda^\circ(S_{k+1} = j-1 \mid S_k = j).
\end{equation*}
According to \eqref{eq:comparewalks:3}, we have
\begin{equation*}
\Prob_\lambda^\circ(Z_{k+1} = j-1 \mid Z_k = j) \leq \tfrac4{10}
\end{equation*}
for $\lambda > 0$ sufficiently small.
This means that we may couple the walks $(S_k)_{k \in \N_0}$ and $(Z_k)_{k \in \N_0}$
such that $\{Z_k-Z_{k-1} = -1\} \subseteq \{S_k-S_{k-1} = -1\}$ for all $k \in \N$.
Define
\begin{equation*}
\varrho \defeq \inf\{S_k: k \in \N \text{ with } S_i < S_k \leq S_j \text{ for all } 0 \leq i < k \leq j\}.
\end{equation*}
A moment's thought reveals that $\varrho \geq \varrho^*$ and hence, for every $n \in \N_0$, by Lemma \ref{Lem:aux result biased RW Z},
\begin{equation*}
\Prob_\lambda^\circ(\varrho^* \geq \lfloor\varepsilon \lambda n\rfloor) \leq \Prob_\lambda^\circ(\varrho \geq \lfloor\varepsilon \lambda n\rfloor)
\leq C^* e^{-c^* \lfloor\varepsilon \lambda n\rfloor} \leq C^* e^{c^*} e^{-c^* \varepsilon \lambda n}.
\end{equation*}
This completes the proof of \eqref{eq:uniform exponential regeneration point estimates}.
\end{proof}

\begin{lemma}	\label{lem:boundtau}
We have
\begin{equation}\label{eq:lem:boundtau1}
\limsup_{\lambda\to 0} \lambda^{4} \E_{\lambda}[(\tau_{1}^{\lambda})^{2}]<\infty \mbox{ and }\limsup_{\lambda\to 0} \lambda^{4} \E_{\lambda}[(\tau_{2}^{\lambda}-\tau_{1}^{\lambda})^{2}]<\infty,
\end{equation}
and
\begin{equation}\label{eq:lem:lowerboundtau1}
\liminf_{\lambda\to 0} \lambda^{2} \E_{\lambda}[\tau_{1}^{\lambda}]>0 \mbox{ and } \liminf_{\lambda\to 0} \lambda^{2} \E_{\lambda}[\tau_{2}^{\lambda}-\tau_{1}^{\lambda}]>0.
\end{equation}
As a consequence,
\begin{equation}\label{eq:lem:boundtau2}
\limsup_{\lambda\to 0} \lambda^{2} \E_{\lambda}[\tau_{1}^{\lambda}]<\infty \mbox{ and }
\limsup_{\lambda\to 0} \lambda^{2} \E_{\lambda}[\tau_{2}^{\lambda}-\tau_{1}^{\lambda}]<\infty,
\end{equation}
and 
\begin{equation}\label{eq:lem:boundtau3}
\limsup_{\lambda\to 0} \frac{ \E_{\lambda}[(\tau_{1}^{\lambda})^{2}]}{\E_{\lambda}[\tau_{1}^{\lambda}]^{2}}<\infty \mbox{ and }
\limsup_{\lambda\to 0} \frac{ \E_{\lambda}[(\tau_{2}^{\lambda}-\tau_{1}^{\lambda})^{2}]}{\E_{\lambda}[\tau_{2}^{\lambda}-\tau_{1}^{\lambda}]^{2}}<\infty.
\end{equation}
\end{lemma}
\begin{proof}
The uniform bounds in \eqref{eq:lem:boundtau2} follow from \eqref{eq:lem:boundtau1} and Jensen's inequality.
The bounds \eqref{eq:lem:boundtau3} follow from \eqref{eq:lem:boundtau1}  and \eqref{eq:lem:lowerboundtau1}.
Let us first prove \eqref{eq:lem:lowerboundtau1}.
The time spent until the first $\lambda$-regeneration is bounded below by the number of visits to the pre-regeneration points
$R^{\mathrm{pre}}_k$ with $0 \leq k < \lfloor 1/\lambda \rfloor^{-1}$,
the pre-regeneration points between $\mathbf{0}$ and $R^{\mathrm{pre},\lambda}_{1}$.
Fix such $k \in \{0,\ldots,\lfloor1/\lambda \rfloor^{-1}-1\}$ and write $N_k$ for the number of returns of $(Y_n)_{n \geq 0}$ to $R^{\mathrm{pre}}_k$.
We shall give a lower bound for $E_{\omega,\lambda}[N_k]$.
We may assume without loss of generality that $R^{\mathrm{pre}}_k=\bfnull$.
Under $P_{\omega,\lambda}$, the number of returns of the walk $(Y_n)_{n \geq 0}$ to $\bfnull$ is geometric with success probability being the escape probability
\begin{equation*}
P_{\omega,\lambda}(Y_n \not = \bfnull \text{ for all } n \geq 1) = \frac{\Con(\bfnull,\infty)}{e^\lambda+1+e^{-\lambda}},
\end{equation*}
where the identity is standard in electrical network theory, see for instance \cite[Formula (13)]{Axelson-Fisk+H"aggstr"om:2009b}.
Consequently,
\begin{equation*}
E_{\omega,\lambda}[N_k] = \Big(1-\frac{\Con(\bfnull,\infty)}{e^\lambda+1+e^{-\lambda}}\Big) / \Big(\frac{\Con(\bfnull,\infty)}{e^\lambda+1+e^{-\lambda}}\Big)
\geq \Res(\bfnull,\infty)
\end{equation*}
for all sufficiently small $\lambda > 0$. From the Nash-Williams inequality \cite[Proposition 9.15]{Levin+Peres+Wilmer:2009},
we infer
\begin{equation*}
\Res(\bfnull,\infty) \geq \sum_{k=1}^\infty \big(2 e^{\lambda(2k-1)}\big)^{-1} = \frac{e^{\lambda}}{2} \cdot \frac{1}{1-e^{-2\lambda}},
\end{equation*}
a bound which is independent of $\omega$. Since there are $\lfloor 1/\lambda\rfloor$ such pre-regeneration points to the left of $R^{\mathrm{pre},\lambda}_{1}$,
we conclude that
\begin{equation*}
\liminf_{\lambda\to 0} \lambda^{2} \E_{\lambda}[\tau_{1}^{\lambda}]
\geq \liminf_{\lambda\to 0} \lambda^{2} \lfloor \tfrac{1}{\lambda}\rfloor \frac{e^{\lambda}}{2} \cdot \frac{1}{1-e^{-2\lambda}} > 0. 
\end{equation*}
This proves the first part in \eqref{eq:lem:lowerboundtau1}.
The second part is analogous or follows using Lemma \ref{Lem:lambda-regeneration times and points}.

Let us turn to \eqref{eq:lem:boundtau1}.
We shall prove the unconditioned case for $\tau_{1}^{\lambda}$;
the conditioned case involving $\tau_{2}^{\lambda}-\tau_{1}^{\lambda}$ follows similarly.
We again use the decomposition:
\begin{equation}\label{eq:tau_1 decomposed2}
\tau_1^{2} = (\tau_1^{\lambda,\back} + \tau_1^{\lambda, \mathrm{traps}})^{2}.
\end{equation}
First we treat $\E_{\lambda}[(\tau_1^{\mathrm{traps}})^{2}]$. 

In order to control the time spent in traps we first bound the time spent in a fixed trap of finite length.
Unfortunately the upper bound given in Lemma 6.1(b) in \cite{Gantert+al:2018} is too rough. However, we follow the arguments there but only consider $\kappa=4$.
Let us  consider a discrete line segment $\{0,\ldots,m\}, m\geq 2,$ and a nearest-neighbor random walk $(S_n)_{n \geq 0}$
on this set starting at $i \in \{0,\ldots,m\}$ with transition probabilities
\begin{equation*}
\Probi(S_{k+1}=j+1 \mid S_k=j) =	1-\Probi(S_{k+1}=j-1 \mid S_k=j) = \frac{e^{\lambda}}{e^{-\lambda}+e^{\lambda}}
\end{equation*}
for $j=1,\ldots,m-1$ and
\begin{equation*}
\Probi(S_{k+1}=1 \mid S_k=0) = \Probi(S_{k+1}=m-1 \mid S_k=m) = 1.
\end{equation*}
For $i=0$, we are interested in $\tau_m \defeq \inf\{k \in \N: S_k = 0\}$, the time until the first return of the walk to the origin.
Let $(Z_n)_{n \geq 0}$ be the agile version of $(Y_n)_{n \geq 0}$, i.e., the walk one infers after deleting all entries $Y_n$ for which $Y_n = Y_{n-1}$
from the sequence $(Y_0,Y_1,\ldots)$.
The stopping times $\tau_m$ will be used to estimate the time the agile walk $(Z_n)_{n \geq 0}$ spends
in a trap of length $m$ given that it steps into it.

Let $V_{i} \defeq \sum_{k=1}^{\tau_m-1} \1_{\{S_k=i\}}$ be the number of visits to the point $i$ before the random walk returns to $0$, $i=1,\ldots,m$.
Then $\tau_m = 1+\sum_{i=1}^m V_i$ and, by Jensen's inequality,
\begin{equation}	\label{eq:Etau_m^4 bound}
\Erm_0 [\tau_m^{4}]
~=~	\Erm_0 \bigg[\bigg(1+\sum_{i=1}^m V_i\bigg)^{\!\!4}\bigg]
~\leq~	(m+1)^{3} \bigg(1 + \Erm_0 \bigg[\sum_{i=1}^m V_i^{4}\bigg] \bigg).
\end{equation}
For $i=0,\ldots,m$, let
\begin{equation*}
\sigma_i \defeq \inf\{k \in \N: S_k = i\}
\quad	\text{and}	\quad
r_i \defeq \Probi(\sigma_i < \sigma_0).
\end{equation*}
Given $S_0=i$, when $S_1 = i+1$, then $\sigma_i < \sigma_0$.
When the walk moves to $i-1$ in its first step,
it starts afresh there and hits $i$ before $0$ with probability $\mathrm{P}_{\!\mathit{i}-1}(\sigma_i < \sigma_0)$.
Determining $\mathrm{P}_{\!\mathit{i}-1}(\sigma_i < \sigma_0)$ is the classical ruin problem, hence
\begin{equation}	\label{eq:r_i}
r_i
~=~
\begin{cases}
\frac{e^{\lambda}}{e^{-\lambda}+e^{\lambda}} + \frac{e^{-\lambda}}{e^{-\lambda}+e^{\lambda}} \big(1- \frac{e^{2\lambda}-1}{1-e^{-2\lambda i}} e^{-2\lambda i}\big)
&	\text{for } i=1,\ldots,m-1;	\\
1- \frac{e^{2\lambda}-1}{1-e^{-2\lambda m}} e^{-2\lambda m}
&	\text{for } i=m.
\end{cases}
\end{equation}
In particular, for $i=1,\ldots,m-1$, $r_i$ does not depend on $m$.
Moreover, we have $r_1 \leq r_2 \leq \ldots \leq r_{m-1}$ and $r_1 \leq r_m \leq r_{m-1}$.
By the strong Markov property, for $k \in \N$,
$\Probnull(V_i = k)
= \Probnull(\sigma_i < \sigma_0) \, r_i^{k-1} (1-r_i)$
and hence
\begin{align*}
\Erm_0 [V_i^{4}]
= \sum_{k \geq 1} k^{4} \Probnull(V_i = k)
\leq (1\!-\!r_{m-1}) \sum_{k \geq 1} k^{4} r_{m-1}^{k-1}
\leq p_3(r_{m-1}) \Big( \frac1{1\!-\!r_{m-1}}\Big)^{\!4}\!\!,
\end{align*}
for some polynomial $p_3(x)$ of degree $3$ in $x$ independent of $\lambda$.
Letting $\lambda \to 0$, we find
\begin{equation*}
\lim_{\lambda \to 0} r_{m-1} = \frac{2m-3}{2m-2}
\end{equation*}
and hence
\begin{equation*}
\limsup_{\lambda\to 0} \Erm_0 [V_i^{4}] \leq  p_3\Big(\frac{2m-3}{2m-2}\Big) (2m-2)^{4}.
\end{equation*}
Hence, using equation \eqref{eq:Etau_m^4 bound},
\begin{equation}\label{eq:Etau_m^4 bound2}
\limsup_{\lambda\to 0} \Erm_0 [\tau_m^{4}] \leq  \lim_{\lambda\to 0}\bigg( (m+1)^{3} \bigg(1 + \Erm_0 \bigg[\sum_{i=1}^m V_i^{4}\bigg] \bigg)\bigg) \leq \tilde p(m),
\end{equation}
for some polynomial $\tilde p$.
Let $\ell_1$ denote the length of the trap with the trap entrance having the smallest nonnegative $\x$-coordinate.
Let $\ell_0$ and $\ell_2$ be the lengths of the next trap to the left and right, respectively, etc.
The law of $\ell_0$ differs from the law of the other $\ell_n$ but this difference is not significant for our estimates, see Lemma 5.1 in  \cite{Gantert+al:2018}.
We proceed as in  the proof of Lemma 6.2 in \cite{Gantert+al:2018}.
For any $\omega \in \Omega^*$ and any $v$ on the backbone,
by the same argument that leads to (24) in \cite{Axelson-Fisk+H"aggstr"om:2009b},
\begin{equation}	\label{eq:lower bound for escape probability}	\textstyle
P_{\omega,\lambda}^v(Y_n \not = v \text{ for all } n \in \N)
\geq \frac{(\sum_{k=0}^{\infty} e^{-\lambda k})^{-1}}{e^{\lambda}+1+e^{-\lambda}}
= \frac{1-e^{-\lambda}}{e^{\lambda}+1+e^{-\lambda}}
\eqdef \pesc.
\end{equation}
This bound is uniform in the environment $\omega \in \Omega^*$ but depends on $\lambda$. Denote by $v_{i}$ the entrance of the $i$th trap.  
By the strong Markov property,
$T_i$, the time spent in the $i$th trap, can be decomposed into $M$ i.i.d.~excursions into the trap:
$T_i = T_{i,1}+ \ldots + T_{i,M}$. Since $v_i$ is forwards-communicating, \eqref{eq:lower bound for escape probability} implies
that  $P_{\omega,\lambda}(M \geq n) \leq (1-\pesc)^{n}$, $n \in \N$.
In particular, $M$ is stochastically bounded by a geometric random variable $\tilde M$ with success parameter $\pesc$.
Moreover, $T_{i,1}, \ldots, T_{i,j}$ are i.i.d.~conditional on $\{M \geq j\}$.
We now derive an upper bound for $E_{\omega,\lambda}[T_{i,j}^{4} | M \geq j]$. To this end, we have to take into account the times the walk stays put.
Each time, the agile walk $(Z_n)_{n \geq 0}$ makes a step in the trap,
this step is preceded by an independent geometric number of times the lazy walk stays put.
The success parameter of this geometric random variable depends on the position inside the trap.
However, it is stochastically dominated by a geometric random variable $G$ with
$\mathrm{P}_{\!0}(G \geq k) = \gamma_\lambda^k$ for $\gamma_\lambda = (1+e^{\lambda})/(e^{\lambda}+1+e^{-\lambda})$.
Plainly, $\gamma_\lambda \to \frac23$ as $\lambda \to 0$. Consequently, estimate \eqref{eq:Etau_m^4 bound2} and Jensen's inequality give
\begin{equation*}
\limsup_{\lambda\to 0}E_{\omega,\lambda}[T_{i,j}^{4} | M \geq j] \leq \limsup_{\lambda\to 0} \Erm_0[\tau_{\ell_i}^4] \Erm_0[G^4] \leq \hat p(\ell_{i}),
\end{equation*}
where $\ell_{i}$ is the length of the $i$th trap (which is treated as a constant under the expectation $\Erm_0$) and $\hat p = \Erm_0[G] \cdot \tilde p$ is again a polynomial.
Moreover, by Jensen's inequality and the strong Markov property,
\begin{align*}
E_{\omega,\lambda}[T_i^{4}]
&=  \sum_{m=1}^{\infty}  
E_{\omega,\lambda}\bigg[\bigg(\sum_{j=1}^m T_{i,j}\bigg)^{\!\!4} \bigg| M = m \bigg] P_{\omega,\lambda}(M=m) \\
& \leq \sum_{m=1}^\infty m^{4}  E_{\omega,\lambda}[T_{i,1}^{4} | M\geq 1] P_{\omega,\lambda}(M = m) \\
& \leq \Erm [\tilde M^4]  E_{\omega,\lambda}[T_{i,1}^{4} | M\geq 1]	\\
& \leq \tilde c \Big( \frac1{\pesc}\Big)^{4} E_{\omega,\lambda}[T_{i,1}^{4} | M\geq 1],
\end{align*}
for some constant $\tilde c$ independent of $\omega$ and $\lambda$.
We have
\begin{equation}	\label{eq:order of 1/pesc}
\limsup_{\lambda\to 0} \lambda  \Big( \frac1{\pesc}\Big)<\infty.
\end{equation}
Hence
\begin{equation}\label{eq:unifbournTi}
\limsup_{\lambda\to 0} \lambda^{4}   \E_{\lambda}[T_{i}^{4} | \ell_{i}=m] \leq p^*(m),
\end{equation}
where $p^*$ is a polynomial with coefficients independent of $\lambda$. Using Lemma 3.5 in \cite{Gantert+al:2018}, we find
\begin{align*}
\E_{\lambda}[T_{i}^{4}]
&=	\sum_{m\geq 1} \Prob_{\lambda}(\ell_{i}=m) \E_{\lambda}[T_{i}^{4} | \ell_{i}=m] \cr
& \leq	c(p) \sum_{m\geq 1} m e^{-2\lambdacrit m}  \E_{\lambda}[T_{i}^{4} | \ell_{i}=m],
\end{align*}
where $c(p)$ is a constant only depending on $p$.
Due to \eqref{eq:unifbournTi}, the dominated convergence theorem applies and gives the following bound:
\begin{equation}\label{eq:lambdaboudT}
\limsup_{\lambda\to 0} \lambda^{4} \E_{\lambda}[T_{i}^{4}] < \infty.
\end{equation}
Now let $L\defeq -\min\{X_{k}, k\in \N\}$ be the absolute value of the leftmost visited $\x$-coordinate of the walk.
Since
\begin{equation}\label{eq:decomposition tau traps}
\E_{\lambda}\big[ (\tau_{1}^{\lambda, \mathrm{traps}})^{2}\big]
\leq \E_{\lambda}\bigg[ \bigg(\sum_{i=-L}^{-1} T_{i} + T_{0} + \sum_{i=1}^{\rho_{1}^{\lambda}} T_{i} \bigg)^{\!\! 2}\bigg]
\end{equation}
we first consider
\begin{align}\label{eq:decouplingTandrho1}
\E_{\lambda}\bigg[ \bigg(\sum_{i=1}^{\rho_{1}^{\lambda}} T_{i}\bigg)^{\!\! 2}\bigg]
&=	\E_{\lambda}\bigg[\sum_{i,j=1}^\infty  T_{i} T_j \1_{\{\rho_{1}^\lambda \geq i \vee j\}}   \bigg]	\notag	\\
&\leq	 \sum_{j=1}^\infty \E_{\lambda}\big[T_{j}^2 \1_{\{\rho_{1}^\lambda \geq j\}}\big]
+ 2 \sum_{j=1}^\infty \sum_{i=1}^{j-1} \E_{\lambda}\big[T_{i} T_{j} \1_{\{\rho_{1}^\lambda \geq j\}}\big].
\end{align}
One application of the Cauchy-Schwarz inequality for the first sum and two applications for the second give
\begin{align*}
\E_{\lambda}\bigg[ \bigg(\sum_{i=1}^{\rho_{1}^{\lambda}} T_{i}\bigg)^{\!\! 2}\bigg]
&\leq \sum_{j=1}^\infty (\E_{\lambda}[T_{j}^4])^{1/2} \Prob_\lambda(\rho_{1}^\lambda \geq j)^{1/2}	\\
&\hphantom{\leq} + 2 \sum_{j=1}^\infty \sum_{i=1}^{j-1} (\E_{\lambda}[T_{i}^4] \E_{\lambda}[T_{j}^4])^{1/4} \Prob_\lambda(\rho_{1}^\lambda \geq j)^{1/2}	\\
&\leq (\E_{\lambda}[T_{1}^4])^{1/2} \bigg( \sum_{j=1}^\infty  C^{1/2} e^{- \lambda \varepsilon j/2} + 2 \sum_{j=1}^\infty j \Prob_\lambda(\rho_{1}^\lambda \geq j)^{1/2} \bigg).
\end{align*}
With the estimates \eqref{eq:lambdaboudT} and \eqref{eq:asymptotics after Cauchy-Schwarz trick} we obtain
\begin{equation*}
\limsup_{\lambda\to 0} \lambda^{4} \E_{\lambda}\bigg[ \bigg(\sum_{i=1}^{\rho_{1}^{\lambda}} T_{i}\bigg)^{\!\!2}\bigg]<\infty.
\end{equation*}
The first term in the upper bound in \eqref{eq:decomposition tau traps} is treated in the same way.
Next, we show that $\Prob_\lambda(L \geq m)$ decays exponentially fast in $m$.
Indeed, $L \geq 2m$ implies that there is an excursion on the backbone to the left of length at least $m$
or the origin is in a trap that covers the piece $[-m,0)$ and thus has length at least $m$.
The probability that there is an excursion on the backbone of length at least $m$ is bounded by a constant (independent of $\lambda$)
times $e^{-2\lambda m}$ by Lemma 6.3 in \cite{Gantert+al:2018}.
The probability that a trap that covers the piece $[-m,0)$ is bounded by a constant (again independent of $\lambda$) times $e^{-2\lambdacrit m}$
by \cite[pp.\;3403--3404]{Axelson-Fisk+H"aggstr"om:2009} or \cite[Lemma 3.2]{Gantert+al:2018}.
We may thus argue as above to conclude that
\begin{equation*}
\limsup_{\lambda\to 0} \lambda^{4} \E_{\lambda}\bigg[ \bigg(\sum_{i=-L}^{-1} T_{i}\bigg)^{\!\!2}\bigg]<\infty.
\end{equation*}
Regarding the term $\E_{\lambda}[T_{0}^{2}]$, we can apply \eqref{eq:lambdaboudT}.
Controlling the mixed terms in 
\eqref{eq:decomposition tau traps} using the Cauchy-Schwarz inequality we obtain
\begin{equation}\label{eq:tau traps}
\limsup_{\lambda\to 0}\lambda^{4} \E_{\lambda}\big[(\tau_{1}^{\lambda, \mathrm{traps}})^{2}\big] <\infty.
\end{equation}
Next we treat the time on the backbone. Since the strategy of proof is the same as for the traps we try to be as brief as possible.
Write $N(v) \defeq \sum_{n \geq 0} \1_{\{Y_n=v\}}$ for the number of visits of the walk $(Y_n)_{n \geq 0}$ to $v \in V$.
We have
\begin{align}
\E_{\lambda}\big[ \big(\tau_1^{\lambda, \back}\big)^{2}\big]
&\leq \E_{\lambda}\bigg[  \bigg(\sum_{-L \leq \x(v) \leq \rho_{1}^{\lambda}} N(v) \1_{\{v \in \back\}}\bigg)^{\!\!2}\bigg]	\notag	\\
&= \E_{\lambda}\bigg[  \bigg(\sum_{-L \leq \x(v) < 0} N(v) \1_{\{v \in \back\}} + \sum_{0 \leq \x(v) \leq \rho_{1}^{\lambda}} N(v) \1_{\{v \in \back\}}\bigg)^{\!\!2}\bigg]	\label{eq:backbone 2nd moment}
\end{align}
We treat the second moment of the second sum first. Using the Cauchy-Schwarz inequality twice, we infer
\begin{align*}
&\E_{\lambda}\bigg[\bigg(\sum_{0 \leq \x(v) \leq \rho_{1}^{\lambda}} N(v) \1_{\{v \in \back\}}\bigg)^{\!\!2}\bigg]	\\
&=	\E_{\lambda}\bigg[\sum_{\x(v),\x(w) \geq 0} N(v) N(w) \1_{\{v,w \in \back\}} \1_{\{\rho_{1}^{\lambda} \geq \x(v) \vee \x(w)\}}\bigg]	\\
&\leq 2 \E_{\lambda}\bigg[\sum_{j=0}^\infty \sum_{i=0}^j \sum_{\substack{\x(v)=i, \\\x(w)=j}} N(v) N(w) \1_{\{v,w \in \back\}} \1_{\{\rho_{1}^{\lambda} \geq j\}}\bigg]\\
&\leq 2 \sum_{j=0}^\infty \sum_{i=0}^j \sum_{\substack{\x(v)=i, \\\x(w)=j}} \big(\E_{\lambda} [N(v)^4 \1_{\{v \in \back\}}]\big)^{1/4} \big(\E_\lambda[N(w)^4 \1_{\{w \in \back\}}]\big)^{1/4}
\Prob_\lambda(\rho_{1}^{\lambda} \geq j)^{1/2} \!.
\end{align*}
The number of visits to $v \in \back$ is stochastically dominated by a geometric random variable with success probability $\pesc$, see \eqref{eq:lower bound for escape probability}.
Hence
\begin{equation*}
\E[N(v)^4 \1_{\{v \in \back\}}] \leq \tilde c \Big( \frac1{\pesc}\Big)^{4}.
\end{equation*}
Using \eqref{eq:asymptotics after Cauchy-Schwarz trick} and \eqref{eq:order of 1/pesc}, we infer
\begin{align*}
\limsup_{\lambda \to 0} & \lambda^4 \E_{\lambda}\bigg[\bigg(\sum_{0 \leq \x(v) \leq \rho_{1}^{\lambda}} N(v) \1_{\{v \in \back\}}\bigg)^{\!\!2}\bigg]	\\
&\leq 8 \tilde c^{1/2} \limsup_{\lambda \to 0} \lambda^4 \Big( \frac1{\pesc}\Big)^{2} \sum_{j=0}^\infty (j+1) \Prob_\lambda(\rho_{1}^{\lambda} \geq j)^{1/2} < \infty.
\end{align*}
We may argue similarly to infer the analogous statement for the first sum in \eqref{eq:backbone 2nd moment}.
Hence, using again the Cauchy-Schwarz inequality for the mixed terms in \eqref{eq:backbone 2nd moment}, we conclude that
\begin{equation}\label{eq:tau backbone}
\limsup_{\lambda\to 0}\lambda^{4} \E_{\lambda}\left[ (\tau_{1}^{\lambda, \back})^{2}   \right] <\infty.
\end{equation}
Using the Cauchy-Schwarz inequality for the mixed terms in decomposition \eqref{eq:tau_1 decomposed2}
together with \eqref{eq:tau traps} and \eqref{eq:tau backbone}, we finally obtain the first statement in \eqref{eq:lem:boundtau1}.
The second statement in  \eqref{eq:lem:boundtau1} then follows from Lemma \ref{Lem:lambda-regeneration times and points}.
\end{proof}

The existence of a regeneration structure allows to express the linear speed in terms of regeneration points and times. 
\begin{lemma}	\label{Lem:SLLN}
Let $\lambda>0$. Then
\begin{equation*}
\vel(\lambda) =  \frac{  \E_{\lambda}[\rho_{2}^{\lambda} - \rho_{1}^{\lambda}]}{ \E_{\lambda}[\tau_{2}^{\lambda}-\tau_{1}^{\lambda} ]}.
\end{equation*}
\end{lemma}
We omit the proof as it is standard and can be derived as \cite[Proposition 4.3]{Gantert+al:2018},
with references to classical renewal theory replaced by references to renewal theory for $1$-dependent variables
as presented in \cite{Janson:1983}.
As a consequence of Lemmas \ref{Lem:uniform exponential regeneration point estimates}, \ref{lem:boundtau} and \ref{Lem:SLLN}, we obtain
\begin{equation}	\label{eq:limsup v(lambda)/lambda}
\limsup_{\lambda \to 0} \frac{\vel(\lambda)}{\lambda} < \infty.
\end{equation}

\subsection{Proof of Equation \eqref{eq:2nd step}.}	\label{subsec:final step}

It remains to prove \eqref{eq:2nd step}, i.e.,
\begin{equation*}	\tag{\ref{eq:2nd step}}
\lim_{\alpha\to\infty} \lim_{\substack{\lambda \to 0,\\ \lambda^{2} n \to \alpha}}
\bigg[\frac{\vel(\lambda)}{\lambda} - \frac{\E_{\lambda}[X_n]}{\lambda n}\bigg]
~=~ 0.
\end{equation*}
The proof follows along the lines of Section 5.3 in \cite{Gantert+al:2012}.
In order to keep this paper self-contained, we repeat the corresponding arguments from \cite{Gantert+al:2012} in the present context.

For $\lambda>0$, we set
\begin{equation*}
k(n)	\defeq	 \left\lfloor \frac{n}{\E_{\lambda}[\tau^\lambda_2-\tau^\lambda_1]} \right\rfloor,	\quad	n \in \N.
\end{equation*}
Notice that $k(n)$ is deterministic but depends on $\lambda$ even though this dependence does not figure in the notation.
Analogously, we shall sometimes write $\tau_n$ for $\tau_n^\lambda$ and, thereby, suppress the dependence on $\lambda$.
For the proof of \eqref{eq:2nd step}, it is sufficient to show that
\begin{align}
\lim_{\alpha \to \infty} \limsup_{\substack{\lambda \to 0,\\ \lambda^2 n \sim \alpha}}
\bigg| \frac{\E_{\lambda}[X_{\tau_{k(n)}}]}{\lambda n} - \frac{\vel(\lambda)}{\lambda}\bigg|
&~=~	0	\label{eq:vel<->reg pts}\\
\text{and}	\qquad
\lim_{\alpha \to \infty} \limsup_{\substack{\lambda \to 0,\\ \lambda^2 n \sim \alpha}}
\bigg| \frac{1}{\lambda n} \E_{\lambda}[X_n-X_{\tau_{k(n)}}]\bigg|	\label{eq:X_n<->X_tau_k(n)}
&~=~	0.
\end{align}
For the proof of \eqref{eq:vel<->reg pts}, notice that
\begin{align}
\bigg| \frac{\E_{\lambda}[X_{\tau_{k(n)}}]}{\lambda n} - \frac{\vel(\lambda)}{\lambda}\bigg|
&=
\bigg| \frac1{\lambda n} \E_{\lambda}\bigg[\sum_{j=1}^{k(n)} (X_{\tau_j}-X_{\tau_{j-1}}) \bigg] - \frac{\vel(\lambda)}{\lambda}\bigg|	\notag	\\
&\leq \frac{1}{\lambda n} \E_{\lambda}[X_{\tau_1}]
+ \bigg| \frac1{\lambda n} (k(n)-1) \E_{\lambda}[X_{\tau_2}-X_{\tau_{1}}] - \frac{\vel(\lambda)}{\lambda}\bigg|.	\label{eq:vel<->reg pts2}
\end{align}
Here, using that $\lambda^2 n \sim \alpha$,
we have that $\frac{1}{\lambda n} \E_{\lambda}[X_{\tau_1}] \sim \alpha^{-1} \lambda \E_{\lambda}[X_{\tau_1}]$.
Thus, the first term in \eqref{eq:vel<->reg pts2} vanishes as first $\lambda \to 0$ and $n \to \infty$ simultaneously
and then $\alpha \to \infty$ by \eqref{eq:lambda^2 E_lambda X_tau_1^2<infty}.
Turning to the second summand in \eqref{eq:vel<->reg pts2}, we first notice that by Lemma \ref{Lem:SLLN},
we have $\vel(\lambda) = \E_{\lambda}[X_{\tau_2}-X_{\tau_1}] / \E_{\lambda}[\tau_2-\tau_1]$ and hence
\begin{align*}
\bigg| & \frac1{\lambda n} (k(n)-1) \E_{\lambda}[X_{\tau_2}-X_{\tau_{1}}] - \frac{\vel(\lambda)}{\lambda}\bigg|	\\
&= \frac{\vel(\lambda)}{\lambda} \bigg| \frac{\E_{\lambda}[\tau_2-\tau_{1}]}{n} \Big(\left\lfloor \frac{n}{\E_{\lambda}[\tau_2-\tau_1]} \right\rfloor-1\Big) - 1\bigg|.
\end{align*}
This expression vanishes as $\lambda \to 0$ since $\limsup_{\lambda \to 0} \frac{\vel(\lambda)}{\lambda} < \infty$ by \eqref{eq:limsup v(lambda)/lambda}.

It finally remains to prove \eqref{eq:X_n<->X_tau_k(n)}.
We begin by proving the analogue of Lemma 5.13 in \cite{Gantert+al:2012}.
While the proof is essentially the same, we have to replace the independence property of the times between two regenerations by the $1$-dependence property.
\begin{lemma}	\label{eq:tau_k-kE_[tau_2-tau_1]}
For all $\varepsilon >0$,
\begin{equation*}
\Prob_{\lambda}\big(| \tau_{k} \!-\! k \E_{\lambda}[\tau_{2} \!-\!\tau_{1}]| \geq \varepsilon  k \E_{\lambda}[\tau_{2}\!-\!\tau_{1}]\big) \to 0 \text{ for } k\to \infty
\end{equation*}
uniformly in $\lambda \in (0,\lambda^*]$ for some sufficiently small $\lambda^* > 0$.
\end{lemma}
\begin{proof}
An application of Markov's inequality yields
\begin{align*}
\Prob_{\lambda} & \big(| \tau_{k} \!-\! k \E_{\lambda}[\tau_{2} \!-\!\tau_{1}]| \geq \varepsilon  k \E_{\lambda}[\tau_{2}\!-\!\tau_{1}]\big)	\\
&\leq \frac1{\varepsilon^2 k^2 (\E_{\lambda}[\tau_{2}\!-\!\tau_{1}])^2} \E_\lambda \big[\big| \tau_{k}\!-\!k\E_{\lambda}[\tau_{2}\!-\!\tau_{1}]\big|^2\big]	\\
&= \frac1{\varepsilon^2 k^2\E_{\lambda}[\tau_{2}\!-\!\tau_{1}]^2}
\E_\lambda \bigg[\!\bigg( \! \tau_1 \!-\! \E_\lambda[\tau_2\!-\!\tau_1]
+\sum_{j=2}^k (\tau_{j}\!-\!\tau_{j-1} \!-\! \E_{\lambda}[\tau_{2} \!-\!\tau_{1}]) \! \bigg)^{\!\!2}\bigg].
\end{align*}
Denote the summands under the square by $A_1,\ldots,A_k$.
Expanding the square and using the $1$-dependence of the $A_j$ and the fact that all $A_j$ but $A_1$ are centered, we infer
\begin{align*}
\Prob_{\lambda} & \big(| \tau_{k} \!-\! k \E_{\lambda}[\tau_{2} \!-\!\tau_{1}]| \geq \varepsilon  k \E_{\lambda}[\tau_{2}\!-\!\tau_{1}]\big)	\\
&\leq \frac1{\varepsilon^2 k^2\E_{\lambda}[\tau_{2}\!-\!\tau_{1}]^2}
\bigg(\sum_{j=1}^k \E_\lambda[A_j^2] + \sum_{|i-j|=1} \E_\lambda[A_i A_j] \bigg)	\\
&\leq \frac1{\varepsilon^2 k^2\E_{\lambda}[\tau_{2}\!-\!\tau_{1}]^2}
\bigg(\sum_{j=1}^k \E_\lambda[A_j^2] + 2\big(\E_\lambda[A_1^2] \E_\lambda[A_2^2]\big)^{1/2} \\
&\hphantom{\leq \frac1{\varepsilon^2 k^2\E_{\lambda}[\tau_{2}\!-\!\tau_{1}]^2}
\bigg(}~+ 2(k-1) \big(\E_\lambda[A_2^2] \E_\lambda[A_3^2 ]\big)^{1/2} \bigg).
\end{align*}
The assertion now follows from Lemma \ref{lem:boundtau}.
\end{proof}
Now fix $\varepsilon > 0$ and write
\begin{align}
\frac{1}{\lambda n} \big|\E_{\lambda}[X_n-X_{\tau_{k(n)}}]\big|
&\leq \frac{1}{\lambda n} \big|\E_{\lambda}[(X_n-X_{\tau_{k(n)}}) \1_{\{|\tau_{k(n)}-n| < \varepsilon n\}}]\big|	\notag	\\
&\hphantom{\leq}~+\frac{1}{\lambda n}\big|\E_{\lambda}[(X_n-X_{\tau_{k(n)}})\1_{\{|\tau_{k(n)}-n| \geq \varepsilon n\}}]\big|.	\label{eq:decomposition wrt tau_k(n)}
\end{align}
Using the Cauchy-Schwarz inequality, the first term on the right-hand side of \eqref{eq:decomposition wrt tau_k(n)} can be bounded as follows.
\begin{align*}
\frac{1}{\lambda n}\big|\E_{\lambda}&[(X_n\!-\!X_{\tau_{k(n)}}) \1_{\{|\tau_{k(n)}-n| < \varepsilon n\}}]\big|
\leq \frac{2}{\lambda n} \E_{\lambda}\Big[\max_{|j-n| < \varepsilon n} |X_{j}\!-\!X_{\lfloor(1-\varepsilon)n\rfloor}|\Big]	\\
&= \frac{2}{\lambda n} \E_0\Big[\max_{|j-n| < \varepsilon n} |X_{j}\!-\!X_{\lfloor(1-\varepsilon)n\rfloor}| e^{\lambda M_n - \lambda^2 A_n + R_{\lambda,n}} \Big]	\\
&\leq \frac{2}{\lambda n} \E_0\Big[\max_{|j-n| < \varepsilon n} |X_{j}\!-\!X_{\lfloor(1-\varepsilon)n\rfloor}|^2 \Big]^{1/2}
\E_0\big[e^{2 \lambda M_n - 2 \lambda^2 A_n + 2 R_{\lambda,n}} \big]^{1/2}.
\end{align*}
We infer $\limsup_{\lambda^2n \to \alpha} \E_0[e^{2 \lambda M_n - 2 \lambda^2 A_n + 2 R_{\lambda,n}}] < \infty$ as in the proof of Proposition \ref{Prop:3rd step}.
Regarding the first factor, we find
\begin{align*}
\frac{1}{\lambda^2 n^2} \E_0\Big[\max_{|j-n| < \varepsilon n} |X_{j}\!-\!X_{\lfloor(1-\varepsilon)n\rfloor}|^2 \Big]
&= \frac{1}{\lambda^2n} \frac1n\Ermp \Big[ E_{\overline{\omega}(\lfloor(1-\varepsilon)n\rfloor),0} \Big[\max_{0  \leq j \leq 2\varepsilon n} X_{j}^2 \Big]\Big]	\\
&= \frac{2\varepsilon}{\lambda^2n} \frac1{2\varepsilon n} \E_0\Big[\max_{0  \leq j \leq 2\varepsilon n} X_{j}^2 \Big].
\end{align*}
This term vanishes as first $\lambda^2 n \to \alpha$ (by Lemma \ref{Lem:max X_k^2 bound}) and then $\alpha \to \infty$.
The second term on the right-hand side of \eqref{eq:decomposition wrt tau_k(n)} can be bounded using the Cauchy-Schwarz inequality, namely,
\begin{align*}
\frac{1}{\lambda n}&\big|\E_{\lambda}[(X_n-X_{\tau_{k(n)}})\1_{\{|\tau_{k(n)}-n| \geq \varepsilon n\}}]\big|	\\
&\leq \frac{1}{\lambda n} \E_{\lambda}[(X_n-X_{\tau_{k(n)}})^2]^{1/2} \cdot \Prob_\lambda(|\tau_{k(n)}-n| \geq \varepsilon n)^{1/2}	\\
&\leq \frac{\sqrt{2}}{\lambda^2 n} \big( \lambda^2 \E_{\lambda}[X_n^2]
+ \lambda^2 \E_\lambda[X_{\tau_{k(n)}}^2]\big)^{1/2} \cdot \Prob_\lambda(|\tau_{k(n)}-n| \geq \varepsilon n)^{1/2}.
\end{align*}
The first factor stays bounded as $\lambda^2n \to \alpha$
whereas the second factor tends to $0$ as $n \to \infty$ and $\lambda^2n \to \alpha$ by Lemma \ref{eq:tau_k-kE_[tau_2-tau_1]}.
Altogether, this finishes the proof of \eqref{eq:X_n<->X_tau_k(n)}.

\begin{acknowledgements}
The research of M.\;Meiners was supported by DFG SFB 878 ``Geometry, Groups and Actions''
and by short visit grant 5329 from the European Science Foundation (ESF)
for the activity entitled `Random Geometry of Large Interacting Systems and Statistical Physics'.
The research was partly carried out during mutual visits of the authors
at Aix-Marseille Universit\'e, Tech\-nische Universit\"at Graz, Technische Universit\"at Darmstadt, Universit\"at Innsbruck,
and Tech\-nische Universit\"at M\"unchen.
Grateful acknowledgement is made for hospitality from all five universities.
\end{acknowledgements}

\begin{appendix}
\section{Auxiliary results from random walk theory}

We use the following result, which may be of interest in its own right.

\begin{lemma}	\label{Lem:moderate deviations estimates}
Let $\xi_1, \xi_2,\ldots$ be random variables
on some probability space with underlying probability measure $\Prm$ (and expectation $\Erm$),
and let $S_n \defeq \xi_1+\ldots+\xi_n$, $n \in \N_0$.
\begin{enumerate}[(a)]
	\item
		Let $\alpha > 0$. If $K$ is an $\N_0$-valued random variable
		with $\Erm[K^{\alpha}]<\infty$ and if $\xi_1,\xi_2,\ldots$ are i.i.d.\ with
		$\Erm[|\xi_1|^{\alpha+1}]<\infty$, then $\Erm[S_K^\alpha]<\infty$.
	\item
		Assume that $\xi_1,\xi_2,\ldots$ are nonnegative and i.i.d.\ under $\Prm$
		with $\Erm[e^{\vartheta \xi_1}] < \infty$ for some $\vartheta > 0$.
		Then, for every $\varepsilon > 0$ and $\delta \in (0,\frac{1}{2})$ 
		there is a random variable $K_1$ with $\Erm[K_1^2] < \infty$ such that, for all $n \in \N$,
		\begin{equation}	\label{eq:bound on xi_n}	\textstyle
		\xi_n \leq K_1 + \varepsilon n^{\frac12+\delta} \text{ a.\,s}.
		\end{equation}
	\item
		Assume that $\xi_1,\xi_2,\ldots$ are i.i.d., centered random variables under $\Prm$
		with $\Erm[e^{\vartheta |\xi_1|}] < \infty$ for some $\vartheta > 0$.
		Then, for every $\varepsilon > 0$ and $\delta \in (0,\frac{1}{2})$ 
		there exists a random variable $K_2$ with $\Erm[K_2^2] < \infty$ such that, for all $n \in \N$,
		\begin{equation}	\label{eq:bound on S_n}	\textstyle
		|S_n| \leq K_2 + \varepsilon n^{\frac12+\delta} \text{ a.\,s}.
		\end{equation}
	\item
		Assume that $(S_n)_{n \in \N_0}$ is a martingale and that there is a constant $C > 0$
		with $\Prm(|\xi_n| \leq C) = 1$ for all $n \in \N$.
		Then, for every $\varepsilon > 0$ and $\delta \in (0,\frac{1}{4})$
		there exists a random variable $K_2$ with $\Erm[K_2^2] < \infty$
		such that \eqref{eq:bound on S_n} holds.
\end{enumerate}
\end{lemma}
\begin{proof}
Assertion (a) follows from \cite[Corollary 1]{Gnedin+Iksanov:2011}.

\noindent
For the proof of (b), fix $\varepsilon > 0$ and $\delta \in (0,\frac12)$. Then define
\begin{equation*}
L \defeq \max\{n \in \N_0: n=0 \text{ or } n \geq 1 \text{ and } \xi_n > \varepsilon n^{\frac12+\delta}\}.
\end{equation*}
For $n \geq 1$, the union bound and Markov's inequality give
\begin{align}	\label{eq:eta_n bound}
\Prm(L \geq n)
&= \Prm(\xi_k > \varepsilon k^{\frac12+\delta} \text{ for some } k \geq n)	\notag	\\
&\leq	\textstyle \sum_{k \geq n} \Prm(\xi_k \geq \varepsilon k^{\frac12+\delta})
\leq \Erm[e^{\vartheta \xi_1}] \sum_{k \geq n} e^{-\vartheta \varepsilon k^{\frac12+\delta}}.
\end{align}
Hence, $\Prm(L \geq n)$ decays faster than any negative power of $n$ as $n \to \infty$.
With $K_1 \defeq S_L$, we have $\Erm[K_1^2]<\infty$ from (a) and, for all $n \in \N$, 
\begin{equation*}
\xi_n \leq K_1 + \varepsilon n^{\frac12+\delta}.
\end{equation*}

\noindent
For the proof of assertion (c), we use 
moderate deviation estimates (see e.g.\;\cite[Theorem 3.7.1]{Dembo+Zeitouni:2010}).
The cited theorem gives
\begin{equation*}
\lim_{n \to \infty} n^{-2\delta} \log \Prm(|S_n| \geq \varepsilon n^{\frac12+\delta}) = -\tfrac{\varepsilon^2}{2 \Varm[\xi_1]}.
\end{equation*}
Hence, for any $c \in (0,\frac{\varepsilon^2}{2 \Varm[\xi_1]})$, we have
\begin{equation*}
\Prm(|S_n| \geq \varepsilon n^{\frac12+\delta}) \leq e^{- c n^{2\delta}}
\end{equation*}
for all sufficiently large $n$. With
$L \defeq \max\{n \in \N_0: |S_n| \geq \varepsilon n^{\frac12+\delta}\}$,
we infer for sufficiently large $n$
\begin{align}
\Prm(L \geq n)
&= \Prm(|S_k| \geq \varepsilon k^{\frac12+\delta} \text{ for some } k \geq n)	\notag	\\
&\leq	\textstyle
\sum_{k \geq n} \Prm(|S_k| \geq \varepsilon k^{\frac12+\delta})
\leq
\sum_{k \geq n} e^{- c k^{2\delta}}.	\label{eq:tail bound L}
\end{align}
In particular, $L$ has finite power moments of all orders.
Now define $K_2 \defeq L \vee (\sum_{j=1}^L |\xi_j|)$.
Then $\Erm[K_2^2] < \infty$ by assertion (a) and, for all $n \in \N$,
\begin{equation}	\label{eq:bound on |S_n|}	\textstyle
|S_n| \leq K_2 + \varepsilon n^{\frac12+\delta}.
\end{equation}

\noindent
Assertion (d) follows from an application of the Azuma-Hoeffding inequality \cite[E14.2]{Williams:1991}.
The cited inequality gives for $L_\pm \defeq \max\{n \in \N_0: \pm S_n \geq \varepsilon n^{\frac12+\delta}\}$
\begin{equation*}
\Prm(L_\pm \geq n)
\leq \sum_{k \geq n} \Prm(\pm S_k \geq \varepsilon k^{\frac12+\delta})
\leq \sum_{k \geq n} e^{-\varepsilon^2 k^{2\delta}/2 C^2}.
\end{equation*}
As above, we conclude that $L \defeq L_+ \vee L_-$ has finite power moments of all orders and,
for all $n \in \N$, \eqref{eq:bound on |S_n|} holds with $K_2 \defeq CL$.
\end{proof}

Finally, we use the following lemma for biased nearest-neighbor random walk on $\Z$.
It is possible that the result is available in the literature. However, we have not been able to locate it.

\begin{lemma}	\label{Lem:aux result biased RW Z}
Let $(S_n)_{n \in \N_0}$ be a biased nearest-neighbor random walk on $\Z$
with respect to some probability measure $\Prm$, i.e.,
$\Prm(S_0=0)=1$ and
\begin{equation*}
\tfrac12 < r \defeq \Prm(S_{n+1} = k+1 \mid S_n = k) = 1-\Prm(S_{n+1} = k-1 \mid S_n = k)
\end{equation*}
for all $k \in \Z$ and $n \in \N_0$. Further, let
\begin{equation*}
\varrho \defeq \inf\{j \in \N: S_i < S_j \leq S_k \text{ for all } 0 \leq i < j \leq k\}
\end{equation*}
be the first positive point the walk visits from which it never steps to the left.
Then there exist finite constants $C^*,c^* = c^*(r) > 0$ such that
$\Prm(\tau \geq k) \leq C^* e^{-c^* k}$ for all $k \in \N_0$.
\end{lemma}
\begin{proof}
The proof is standard and relies on the usual recursive construction of regeneration times,
see e.g.\;\cite{Kesten:1977}, and the Gambler's ruin formula.
We omit the details. 
\end{proof}

\end{appendix}


\bibliographystyle{spmpsci}
\bibliography{RWRE}

\begin{thebibliography}{10}
\providecommand{\url}[1]{{#1}}
\providecommand{\urlprefix}{URL }
\expandafter\ifx\csname urlstyle\endcsname\relax
  \providecommand{\doi}[1]{DOI~\discretionary{}{}{}#1}\else
  \providecommand{\doi}{DOI~\discretionary{}{}{}\begingroup
  \urlstyle{rm}\Url}\fi

\bibitem{Avena+2013}
Avena, L., dos Santos, R.S., V{\"o}llering, F.: Transient random walk in
  symmetric exclusion: limit theorems and an {E}instein relation.
\newblock ALEA Lat. Am. J. Probab. Math. Stat. \textbf{10}(2), 693--709 (2013)

\bibitem{Axelson-Fisk+H"aggstr"om:2009b}
Axelson-Fisk, M., H{\"a}ggstr{\"o}m, O.: Biased random walk in a
  one-dimensional percolation model.
\newblock Stochastic Process. Appl. \textbf{119}(10), 3395--3415 (2009).
\newblock \doi{10.1016/j.spa.2009.06.004}.
\newblock \urlprefix\url{http://dx.doi.org/10.1016/j.spa.2009.06.004}

\bibitem{Axelson-Fisk+H"aggstr"om:2009}
Axelson-Fisk, M., H{\"a}ggstr{\"o}m, O.: Conditional percolation on
  one-dimensional lattices.
\newblock Adv. in Appl. Probab. \textbf{41}(4), 1102--1122 (2009).
\newblock \doi{10.1239/aap/1261669588}.
\newblock \urlprefix\url{http://dx.doi.org/10.1239/aap/1261669588}

\bibitem{BenArous+Hu+Olla+Zeitouni:2013}
Ben~Arous, G., Hu, Y., Olla, S., Zeitouni, O.: Einstein relation for biased
  random walk on {G}alton-{W}atson trees.
\newblock Ann. Inst. Henri Poincar\'e Probab. Stat. \textbf{49}(3), 698--721
  (2013).
\newblock \doi{10.1214/12-AIHP486}.
\newblock \urlprefix\url{http://dx.doi.org/10.1214/12-AIHP486}

\bibitem{Berger+Biskup:2007}
Berger, N., Biskup, M.: Quenched invariance principle for simple random walk on
  percolation clusters.
\newblock Probab. Theory Related Fields \textbf{137}(1-2), 83--120 (2007).
\newblock \doi{10.1007/s00440-006-0498-z}.
\newblock \urlprefix\url{http://dx.doi.org/10.1007/s00440-006-0498-z}

\bibitem{Billingsley:1968}
Billingsley, P.: Convergence of probability measures.
\newblock John Wiley \& Sons, Inc., New York-London-Sydney (1968)

\bibitem{DeMasi+al:1989}
De~Masi, A., Ferrari, P.A., Goldstein, S., Wick, W.D.: An invariance principle
  for reversible {M}arkov processes. {A}pplications to random motions in random
  environments.
\newblock J. Statist. Phys. \textbf{55}(3-4), 787--855 (1989).
\newblock \doi{10.1007/BF01041608}.
\newblock \urlprefix\url{http://dx.doi.org/10.1007/BF01041608}

\bibitem{Dembo+Zeitouni:2010}
Dembo, A., Zeitouni, O.: Large deviations techniques and applications,
  \emph{Stochastic Modelling and Applied Probability}, vol.~38.
\newblock Springer-Verlag, Berlin (2010).
\newblock \doi{10.1007/978-3-642-03311-7}.
\newblock \urlprefix\url{http://dx.doi.org/10.1007/978-3-642-03311-7}.
\newblock Corrected reprint of the second (1998) edition

\bibitem{Einstein}
Einstein, A.: Investigations on the theory of the {B}rownian movement.
\newblock Dover Publications, Inc., New York (1956).
\newblock Edited with notes by R. F{\"u}rth, Translated by A. D. Cowper

\bibitem{Faggionato+Salvi:18}
Faggionato, A., Salvi, M.: Regularity of biased 1d random walks in random
  environment.
\newblock arxiv, https://arxiv.org/abs/1802.07874  (2018)

\bibitem{Ferrari++:1985}
Ferrari, P.A., Goldstein, S., Lebowitz, J.L.: Diffusion, Mobility and the
  Einstein Relation, pp. 405--441.
\newblock Birkh{\"a}user Boston, Boston, MA (1985)

\bibitem{Gantert+Guo+Nagel:17}
Gantert, N., Guo, X., Nagel, J.: Einstein relation and steady states for the
  random conductance model.
\newblock Ann. Probab. \textbf{45}(4), 2533--2567 (2017).
\newblock \doi{10.1214/16-AOP1119}.
\newblock \urlprefix\url{https://doi.org/10.1214/16-AOP1119}

\bibitem{Gantert+al:2012}
Gantert, N., Mathieu, P., Piatnitski, A.: Einstein relation for reversible
  diffusions in a random environment.
\newblock Comm. Pure Appl. Math. \textbf{65}(2), 187--228 (2012).
\newblock \doi{10.1002/cpa.20389}.
\newblock \urlprefix\url{http://dx.doi.org/10.1002/cpa.20389}

\bibitem{Gantert+al:2018}
Gantert, N., Meiners, M., M{\"u}ller, S.: Regularity of the {S}peed of {B}iased
  {R}andom {W}alk in a {O}ne-{D}imensional {P}ercolation {M}odel.
\newblock J. Stat. Phys. \textbf{170}(6), 1123--1160 (2018).
\newblock \doi{10.1007/s10955-018-1982-4}.
\newblock \urlprefix\url{https://doi.org/10.1007/s10955-018-1982-4}

\bibitem{Gnedin+Iksanov:2011}
Gnedin, A., Iksanov, A.: Moments of random sums and {R}obbins' problem of
  optimal stopping.
\newblock J. Appl. Probab. \textbf{48}(4), 1197--1199 (2011)

\bibitem{Guo:16}
Guo, X.: Einstein relation for random walks in random environment.
\newblock Ann. Probab. \textbf{44}(1), 324--359 (2016).
\newblock \doi{10.1214/14-AOP975}.
\newblock \urlprefix\url{https://doi.org/10.1214/14-AOP975}

\bibitem{Helland:1982}
Helland, I.S.: Central limit theorems for martingales with discrete or
  continuous time.
\newblock Scand. J. Statist. \textbf{9}(2), 79--94 (1982)

\bibitem{Janson:1983}
Janson, S.: Renewal theory for {$m$}-dependent variables.
\newblock Ann. Probab. \textbf{11}(3), 558--568 (1983)

\bibitem{Kesten:1977}
Kesten, H.: A renewal theorem for random walk in a random environment pp.
  67--77 (1977)

\bibitem{Komorowski+Olla:2005a}
Komorowski, T., Olla, S.: Einstein relation for random walks in random
  environments.
\newblock Stochastic Process. Appl. \textbf{115}(8), 1279--1301 (2005).
\newblock \doi{10.1016/j.spa.2005.03.009}.
\newblock \urlprefix\url{https://doi.org/10.1016/j.spa.2005.03.009}

\bibitem{Komorowski+Olla:2005}
Komorowski, T., Olla, S.: On mobility and {E}instein relation for tracers in
  time-mixing random environments.
\newblock J. Stat. Phys. \textbf{118}(3-4), 407--435 (2005).
\newblock \doi{10.1007/s10955-004-8815-3}.
\newblock \urlprefix\url{https://doi.org/10.1007/s10955-004-8815-3}

\bibitem{Lebowitz+Rost:94}
Lebowitz, J.L., Rost, H.: The {E}instein relation for the displacement of a
  test particle in a random environment.
\newblock Stochastic Process. Appl. \textbf{54}(2), 183--196 (1994).
\newblock \doi{10.1016/0304-4149(94)00015-8}.
\newblock \urlprefix\url{https://doi.org/10.1016/0304-4149(94)00015-8}

\bibitem{Levin+Peres+Wilmer:2009}
Levin, D.A., Peres, Y., Wilmer, E.L.: Markov chains and mixing times.
\newblock American Mathematical Society, Providence, RI (2009).
\newblock With a chapter by James G. Propp and David B. Wilson

\bibitem{Loulakis:2002}
Loulakis, M.: Einstein relation for a tagged particle in simple exclusion
  processes.
\newblock Comm. Math. Phys. \textbf{229}(2), 347--367 (2002).
\newblock \doi{10.1007/s00220-002-0692-5}.
\newblock \urlprefix\url{https://doi.org/10.1007/s00220-002-0692-5}

\bibitem{Loulakis:2005}
Loulakis, M.: Mobility and {E}instein relation for a tagged particle in
  asymmetric mean zero random walk with simple exclusion.
\newblock Ann. Inst. H. Poincar{\'e} Probab. Statist. \textbf{41}(2), 237--254
  (2005).
\newblock \doi{10.1016/j.anihpb.2004.07.001}.
\newblock \urlprefix\url{https://doi.org/10.1016/j.anihpb.2004.07.001}

\bibitem{Mathieu+Piatnitski:2007}
Mathieu, P., Piatnitski, A.: Quenched invariance principles for random walks on
  percolation clusters.
\newblock Proc. R. Soc. Lond. Ser. A Math. Phys. Eng. Sci. \textbf{463}(2085),
  2287--2307 (2007).
\newblock \doi{10.1098/rspa.2007.1876}.
\newblock \urlprefix\url{http://dx.doi.org/10.1098/rspa.2007.1876}

\bibitem{Peligrad+al:2007}
Peligrad, M., Utev, S., Wu, W.B.: A maximal {$\Bbb L_p$}-inequality for
  stationary sequences and its applications.
\newblock Proc. Amer. Math. Soc. \textbf{135}(2), 541--550 (2007).
\newblock \doi{10.1090/S0002-9939-06-08488-7}.
\newblock \urlprefix\url{http://dx.doi.org/10.1090/S0002-9939-06-08488-7}

\bibitem{Rosenblatt:1971}
Rosenblatt, M.: Markov Processes. {S}tructure and asymptotic behavior.
\newblock {S}pringer-{V}erlag, New York (1971).
\newblock Die {G}rundlehren der {M}athematischen {W}issenschaften, {B}and 184

\bibitem{Spohn}
Spohn, H.: Scaling limits for stochastic particle systems.
\newblock In: I{X}th {I}nternational {C}ongress on {M}athematical {P}hysics
  ({S}wansea, 1988), pp. 272--275. Hilger, Bristol (1989)

\bibitem{Williams:1991}
Williams, D.: Probability with martingales.
\newblock Cambridge Mathematical Textbooks. Cambridge University Press,
  Cambridge (1991)

\bibitem{Zeitouni:2004}
Zeitouni, O.: Random walks in random environment.
\newblock In: Lecture notes on probability theory and statistics, \emph{Lecture
  Notes in Math.}, vol. 1837, pp. 189--312. Springer, Berlin (2004)

\end{thebibliography}

\end{document}